\documentclass[11pt,a4paper]{amsart}
\usepackage[utf8]{inputenc}
\usepackage[T1]{fontenc}
\usepackage[a4paper, top=25mm, bottom=30mm, left=25mm, right=25mm]{geometry}
\usepackage{tikz}
\usetikzlibrary{
    graphs,
    topaths,
    arrows,
    shapes,
    shapes.geometric,
    positioning,
    decorations.markings,
    backgrounds,
    matrix,
    calc
}
\usepackage{amsfonts}
\usepackage{mathtools}
\usepackage{amssymb}
\usepackage{graphicx}
\usepackage{amsthm}
\usepackage[english]{babel}
\usepackage{float}
\usepackage{enumitem}
\usepackage{caption}
\usepackage{url}
\usepackage{xcolor}
\tikzstyle arrowstyle=[scale=1]
\tikzstyle reverse directed=[postaction={decorate,decoration={markings,
    mark=at position .65 with {\arrowreversed[arrowstyle]{stealth};}}}]

\tikzset{
    directed/.style={
        draw,
        -,
        postaction={
            decorate,
            decoration={
                markings,
                mark=at position 0.545 with {\arrow[scale=1.4]{latex}}
            }
        }
    }
}

\newcommand{\auttfg}{\Aut^{\mathbf{TF}}(G)}
\newcommand{\Aut}{\mbox{Aut}}
\newcommand{\ADC}{\mbox{{\rm{\textbf{ADC}}}}}
\newcommand{\CDC}{\mbox{{\rm{\textbf{CDC}}}}}
\newcommand{\CG}{\mbox{{\rm{\textbf{CG}}}}}
\newcommand{\opp}{^{\mathrm{opp}}}
\newcommand{\id}{{\rm id}}
\newcommand{\ZZ}{\mathbb Z}

\theoremstyle{plain}
\newtheorem{Thm}{Theorem}[section]
\newtheorem{Prop}[Thm]{Proposition}
\newtheorem{Cor}[Thm]{Corollary}
\newtheorem{Lem}[Thm]{Lemma}
\newtheorem{Con}[Thm]{Construction}
\theoremstyle{definition}
\newtheorem{Def}{Definition}[section]
\theoremstyle{remark}
\newtheorem{Rem}[Thm]{Remark}
\newtheorem{Ex}{Example}[section]

\begin{document}
\setlength{\abovedisplayskip}{5pt plus 2pt minus 1pt}
\setlength{\belowdisplayskip}{5pt plus 2pt minus 1pt}
\setlength{\abovedisplayshortskip}{2pt plus 1pt}
\setlength{\belowdisplayshortskip}{5pt plus 2pt minus 1pt}

\title[Lifting and Folding]{Lifting and Folding:\\
A Framework for Unstable Graphs and TF-Cousins}

\author[R. Mizzi]{Russell Mizzi}
\address{University of Malta \\ Malta}
\email{russell.mizzi@um.edu.mt}


\maketitle

\begin{abstract}
A graph $G$ is \emph{unstable} if its canonical double cover $\CDC(G)$ has
strictly more automorphisms than $\Aut(G)\times\ZZ_2$. A related but
distinct question is when two non-isomorphic graphs share the same CDC.
We place both problems within
a unified framework of \emph{lifting} and \emph{guided folding}, revealing
that each is governed by the same algebraic datum: the conjugacy classes of
strongly switching involutions in $\Aut(\CDC(G))$.

Our approach is based on \emph{two-fold isomorphisms} (TF-isomorphisms), a
generalisation of graph isomorphism, and on lifting and guided folding
adapted from voltage-graph theory. Lifting a TF-isomorphism
$(\alpha,\beta):G\to H$ produces a digraph isomorphic to the alternating
double cover of $G$. Folding this digraph back along a guide
always yields a graph TF-isomorphic to $G$. If the result is non-isomorphic to $G$, then the two
form a TF-cousin pair; if two distinct guides both give a graph
isomorphic to $G$, then $G$ has a non-trivial TF-automorphism and is
unstable. The two problems are thus
aspects of the same construction, addressed simultaneously. Each guide
determines a strongly switching involution in
$\Aut(\CDC(G))$, and distinct
conjugacy classes of such involutions produce non-isomorphic base
graphs sharing the same CDC, a correspondence due to Pacco and
Scapellato.

The framework generates TF-cousin pairs and unstable graphs of
every order from~$6$ on, starting from a seed pair
$(C_k\cup C_k,\, C_{2k})$ for odd $k$. We introduce the
\emph{claw graphs} $\CG(n)$ as a concrete infinite family of examples
and prove that $\CG(n)$ and its companion $\CG'(n)$ are TF-cousins if and
only if $n$ is odd. For $n=1$ the pair $(\CG(1),\CG'(1))$ consists of the
Petersen graph and a companion cubic graph on $10$ vertices, with the Desargues
graph as their common CDC. For each odd $n\geq 3$ the construction yields a
new pair of non-isomorphic cubic graphs sharing the same CDC.
We
conjecture that in every TF-cousin pair one member contains two
vertex-disjoint copies of $C_k$ and the other contains $C_{2k}$, for some
odd~$k$, and that every unstable asymmetric graph contains both $C_k$ and
$C_{2k}$ for some odd~$k$. The first statement has been verified
computationally for all connected graphs on at most $9$ vertices.
\end{abstract}

\section{Preliminaries}
\label{sec:prelim}

We consider simple \emph{mixed graphs}, that is, graphs which may contain
both undirected edges and directed arcs. An undirected edge $\{u,v\}$ in a
mixed graph is treated as the union of two directed arcs, $(u,v)$ and $(v,u)$.
A mixed graph consisting solely of undirected edges is referred to as
a \emph{graph}. Throughout this paper we consider only connected graphs unless
explicitly stated otherwise.
The \emph{converse} (or \emph{opposite}) of a mixed graph $G$,
denoted by $G\opp$, is the mixed graph on the same vertex set obtained by
reversing every arc, so that $(u,v)\in A(G\opp)$ if and only if
$(v,u)\in A(G)$~\cite{bangjensen}. An undirected edge, being the union of
two opposite arcs, is unaffected; thus $G\opp=G$ for every graph, and the
distinction matters only for mixed graphs and digraphs.

Let $G$ and $H$ be two graphs, and let $\alpha$, $\beta$ be bijections from
$V(G)$ to $V(H)$. The pair $(\alpha,\beta)$ is a
\emph{two-fold isomorphism}
(or \emph{TF-isomorphism}) from $G$ to $H$ if, for all $u,v\in
V(G)$, $(u,v)$ is an arc of $G$ if and only if $(\alpha(u),\beta(v))$ is an
arc of $H$. This notion was introduced by
Zelinka~\cite{zelinka2} under the name \emph{isotopy of digraphs} and
later formalised as TF-isomorphism by Lauri et al.~\cite{lms1}. We say that
$G$ and $H$ are \emph{TF-isomorphic} and write
$G\cong^{\mathbf{TF}}H$. When $\alpha=\beta$, this reduces
to an ordinary isomorphism. When $\alpha\neq\beta$, the pair $(\alpha,\beta)$
is a \emph{non-trivial} TF-isomorphism:
$\alpha$ and $\beta$ need not individually be isomorphisms from $G$ to $H$,
as illustrated in Figure~\ref{fig:firstexample}. Two non-isomorphic graphs
that are TF-isomorphic are called \emph{TF-cousins}.

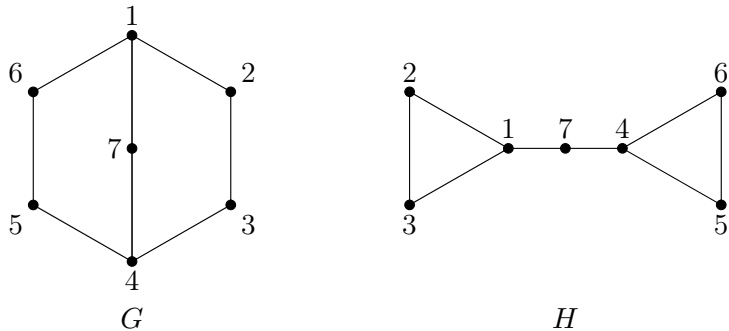
\begin{figure}[h]
\centering
\begin{tikzpicture}[scale=1.5]
\begin{scope}
\draw[-](0,0)--++(-30:1)coordinate(A2)[fill=black]circle(1.2pt)
    node[above right](){2};
\draw[-](A2)--++(-90:1)coordinate(A3)[fill=black]circle(1.2pt)
    node[below right](){3};
\draw[-](A3)--++(-150:1)coordinate(A4)[fill=black]circle(1.2pt)
    node[below](){4};
\draw[-](A4)--++(-210:1)coordinate(A5)[fill=black]circle(1.2pt)
    node[below left](){5};
\draw[-](A5)--++(-270:1)coordinate(A6)[fill=black]circle(1.2pt)
    node[above left](){6};
\draw[-](A6)--++(-330:1)coordinate(A1)[fill=black]circle(1.2pt)
    node[above](){1};
\draw(A1)--(A4)coordinate[midway](A7);
\draw(A7)[fill=black]circle(1.2pt)node[left](){7};
\draw[-](A1)--(A7);
\draw[-](A7)--(A4);
\draw(A7)++(0,-1.5)node[](){$G$};
\end{scope}
\begin{scope}[xshift=3.3cm, yshift=-1cm]
\draw[-](0,0)--++(150:1)coordinate(A2C)[fill=black]circle(1.2pt)
    node[above](){2};
\draw[-](A2C)--++(270:1)coordinate(A3C)[fill=black]circle(1.2pt)
    node[below](){3};
\draw[-](A3C)--++(390:1)coordinate(A1C)[fill=black]circle(1.2pt)
    node[above](){1};
\draw[-](A1C)--++(0:0.5)coordinate(A7C)[fill=black]circle(1.2pt)
    node[above](){7};
\draw[-](A7C)--++(0:0.5)coordinate(A4C)[fill=black]circle(1.2pt)
    node[above](){4};
\draw[-](A4C)--++(30:1)coordinate(A6C)[fill=black]circle(1.2pt)
    node[above](){6};
\draw[-](A6C)--++(-90:1)coordinate(A5C)[fill=black]circle(1.2pt)
    node[below](){5};
\draw[-](A5C)--(A4C);
\draw(A7C)++(0,-1.5)node[](){$H$};
\end{scope}
\end{tikzpicture}
\caption{A non-trivial TF-isomorphism from $G$ to $H$, with
         $\alpha=(2\ 5)$ and $\beta=(1\ 4)(3\ 6)$.}
\label{fig:firstexample}
\end{figure}

When $G=H$, the pair $(\alpha,\beta)$ is a
\emph{TF-automorphism}, which is \emph{non-trivial}
when $\alpha\neq\beta$. The collection of all TF-automorphisms of $G$, with
multiplication $(\alpha,\beta)(\gamma,\delta)=(\alpha\gamma,\beta\delta)$, is
a subgroup of $S_{V(G)}\times S_{V(G)}$ called the
\emph{two-fold automorphism group} of $G$,
denoted by $\auttfg$. Identifying each $\alpha\in\Aut(G)$ with the TF-automorphism
$(\alpha,\alpha)$ gives the inclusion $\Aut(G)\subseteq\auttfg$; equality
holds when $G$ has no non-trivial TF-automorphisms. However,
an \emph{asymmetric} graph $G$ (one with $|\Aut(G)|=1$) may still admit
non-trivial TF-automorphisms~\cite{lms2}.

The \emph{canonical double cover} of $G$, denoted by $\CDC(G)$, has vertex
set
$V(G)\times\ZZ_2$; the pairs $((u,0),(v,1))$ and $((u,1),(v,0))$ are arcs of
$\CDC(G)$ if and only if $(u,v)$ is an arc of~$G$. The CDC coincides with the
direct product $G\times K_2$ \cite{imrich,klin112}
and is always bipartite with colour classes $V_0=V(G)\times\{0\}$ and
$V_1=V(G)\times\{1\}$.

\begin{Def}\label{def:switching}
An automorphism $\phi\in\Aut(\CDC(G))$ is a \emph{switching involution} if
$\phi^2=\id$ and $\phi$ interchanges the two colour classes $V_0$, $V_1$ of
$\CDC(G)$. It is \emph{strongly switching} if, in addition, no edge of
$\CDC(G)$ joins a vertex to its own image under $\phi$, that is,
$\{u,\phi(u)\}\notin E(\CDC(G))$ for every $u\in V(\CDC(G))$.
\end{Def}

It is observed in~\cite{klin112} that $\Aut(\CDC(G))$ equals
$\Aut(G)\times\ZZ_2$ in some cases and properly contains it in others. A
graph is said to be \emph{unstable} if $\Aut(G)\times\ZZ_2$
is a proper subgroup of $\Aut(\CDC(G))$; the elements of
$\Aut(\CDC(G))\setminus(\Aut(G)\times\ZZ_2)$ are called \emph{unexpected
automorphisms} of $\CDC(G)$~\cite{Ars}. The \emph{index of
instability} of $G$ is
\[
    \frac{|\Aut(\CDC(G))|}{2|\Aut(G)|}.
\]
This is a positive integer, equal to $1$ exactly when $G$ is stable.
Stability in this sense was introduced by Maru\v{s}i\v{c},
Scapellato and Zagaglia Salvi~\cite{Scapsalvi1}. Surowski~\cite{Surowski1}
settled it for arc-transitive graphs, and Wilson~\cite{Wilson01} gave
general conditions forcing instability together with families of unstable
edge-transitive graphs.

The following two cases are excluded from consideration:
\begin{itemize}
\item[(a)] If $G$ is bipartite, then $\CDC(G)$ consists of two copies
    of $G$, so it is disconnected and
    $G$ is unstable unless $\Aut(G)$ is trivial.
\item[(b)] If two vertices $u$, $v$ of $G$ share the same neighbourhood
    (that is, $G$ is not vertex-determining), then
    $(\alpha,\id)$ with $\alpha=(u\ v)$ is a non-trivial TF-automorphism.
\end{itemize}
\noindent Unless stated otherwise, every graph we consider is
non-bipartite and vertex-determining.

Standard graph-theoretic terminology follows~\cite{bondy,Harary01};
information on automorphism groups of graphs can be found in~\cite{lauri2}.
In a circuit $C_{2m}$, two vertices at distance $m$ are called
\emph{antipodal}.

\section{Background: Instability and TF-Automorphisms}
\label{sec:unstable}

The following result from~\cite{Ars} is fundamental to both problems studied in this paper.

\begin{Thm}[\cite{Ars}]\label{thm:stabilitytfmain}
Let $G$ be a graph. Then
\[
    \Aut(\CDC(G)) = \Aut^{\mathbf{TF}}(G)\rtimes\ZZ_2.
\]
In particular, $G$ is unstable if and only if it has a non-trivial TF-auto\-morphism.
\end{Thm}

Figure~\ref{fig:gammacockroach} shows the smallest known unstable asymmetric
graph, constructed in~\cite{cockroachpaper} by finding a permutation
$\gamma$ of its vertex set such that $(\gamma,\gamma^{-1})$ is a non-trivial
TF-automorphism. It is the first member of an infinite family of such
graphs.

\begin{figure}[H]
\centering
\begin{tikzpicture}[scale=1.5]
\draw(0,0)coordinate(A6)--++(-60:1)coordinate(A7)
    --++(-180:1)coordinate(A12)--++(60:1);
\draw(A6)--++(-40:2)coordinate(A1);
\draw(A6)--++(220:2)coordinate(A3);
\draw(A1)--++(-90:2)coordinate(A5);
\draw(A3)--++(-90:2)coordinate(A4);
\draw(A5)--++(135:1.25)coordinate(A10);
\draw(A4)--++(45:1.25)coordinate(A9);
\draw(A9)--(A10)coordinate[midway](M);
\draw(M)++(-90:1.9)coordinate(A2);
\draw(A4)--(A2);\draw(A5)--(A2);
\draw(A4)--++(-15:1)coordinate(A11);
\draw(A5)--++(195:1)coordinate(A8);
\draw(A11)--(A8);
\draw(A1)--(A12);\draw(A2)--(A11);\draw(A3)--(A10);
\draw(A1)[fill=black]circle(1.2pt)node[right](){$1$};
\draw(A2)[fill=black]circle(1.2pt)node[below](){$2$};
\draw(A3)[fill=black]circle(1.2pt)node[left](){$3$};
\draw(A4)[fill=black]circle(1.2pt)node[left](){$4$};
\draw(A5)[fill=black]circle(1.2pt)node[right](){$5$};
\draw(A6)[fill=black]circle(1.2pt)node[above](){$6$};
\draw(A7)[fill=black]circle(1.2pt)node[right](){$7$};
\draw(A8)[fill=black]circle(1.2pt)node[above](){$8$};
\draw(A9)[fill=black]circle(1.2pt)node[above](){$9$};
\draw(A10)[fill=black]circle(1.2pt)node[above](){$10$};
\draw(A11)[fill=black]circle(1.2pt)node[above](){$11$};
\draw(A12)[fill=black]circle(1.2pt)node[left](){$12$};
\end{tikzpicture}
\caption{The smallest known unstable asymmetric graph~\cite{cockroachpaper}.
         With $\gamma=(1\ 2\ 3)(4\ 5\ 6)(7\ 8\ 9)(10\ 11\ 12)$, the pair
         $(\gamma,\gamma^{-1})$ is a non-trivial TF-automorphism.}
\label{fig:gammacockroach}
\end{figure}
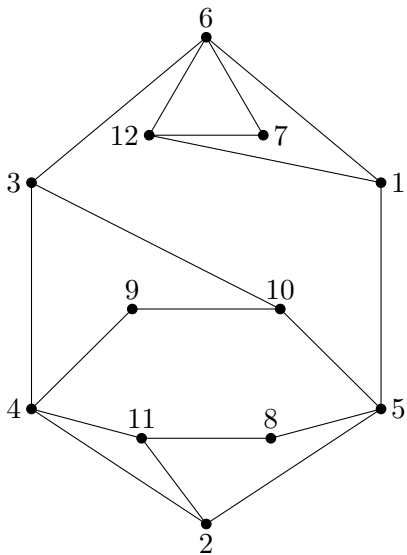

Stability has been studied intensively in recent years,
particularly for circulant and Cayley graphs. Qin, Xia and
Zhou~\cite{qin2019} investigate the stability of circulant graphs, and
Fernandez and Hujdurovi\'{c}~\cite{fernandez2022} study the canonical double
covers of circulants. Morris~\cite{morris2021} studies automorphisms of
direct products of Cayley graphs on abelian groups, Hujdurovi\'{c} and
Mitrovi\'{c}~\cite{hujdurovicmitrovic2024} give conditions implying
stability, and Hujdurovi\'{c} and Kov\'{a}cs~\cite{hujdurovickovacs2025}
treat the stability of Cayley graphs by means of Schur rings. The present
paper is concerned with arbitrary graphs rather than a prescribed
algebraic family, and with the construction of examples rather than
criteria for stability.

\section{\textbf{A}-Trails and Alternating Double Covers}
\label{sec:atrails}

The standard notion of a path carries an implicit direction even for
undirected graphs: under an ordinary isomorphism $\alpha$, the arcs $(u,v)$
and $(v,w)$ map to $(\alpha(u),\alpha(v))$ and $(\alpha(v),\alpha(w))$,
sharing the common vertex $\alpha(v)$. Under a TF-isomorphism $(\alpha,\beta)$,
however, $(u,v)$ maps to $(\alpha(u),\beta(v))$ and $(w,v)$ maps to
$(\alpha(w),\beta(v))$: the shared vertex is $\beta(v)$. Maintaining a common
vertex between the images of successive arcs therefore requires alternating
directions in the original path, motivating the notion of an \textbf{A}-trail
\cite{dissertation}.

A sequence $P=(a_1,a_2,\dots,a_k)$ of arcs in a mixed graph $G$ is an
\emph{alternating trail} (\textbf{A}-\emph{trail}) if each
consecutive pair $a_i$, $a_{i+1}$ shares exactly one vertex, and whenever
$a_i=(x,y)$ either $a_{i+1}=(x,z)$ or $a_{i+1}=(z,y)$ for some vertex~$z$.
The \emph{first} and \emph{last} vertices of $P$ are the vertices of $a_1$
and $a_k$, respectively, that are not shared with their neighbour in the
sequence.

A mixed graph $G$ is \textbf{A}-\emph{connected} if every
pair of vertices is joined by an \textbf{A}-trail. Every connected graph is
\textbf{A}-connected; mixed graphs need not be.

The \emph{alternating double cover} $\ADC(G)$ of a mixed graph $G$ is the
direct product of $G$ with the digraph $\vec{D}$ having $V(\vec{D})=\{0,1\}$ and unique
arc $(0,1)$. Thus $V(\ADC(G))=V(G)\times\{0,1\}$, with
$((u,0),(v,1))\in A(\ADC(G))$ if and only if $(u,v)\in A(G)$; all vertices
$(u,0)$ are sources and all $(u,1)$ are sinks.

\begin{Thm}[\cite{lms1}]\label{thm:cdc}
Graphs $G$ and $H$ are TF-isomorphic if and only if $\CDC(G)\cong\CDC(H)$.
\end{Thm}

\begin{Thm}[\cite{dissertation}]\label{thm:adc}
Mixed graphs $G$ and $H$ are TF-isomorphic if and only if $\ADC(G)\cong\ADC(H)$.
\end{Thm}

Theorem~\ref{thm:adc} extends Theorem~\ref{thm:cdc} to mixed graphs. In the
setting of undirected graphs, which we adopt for the remainder of the paper,
TF-isomorphic graphs $G$ and $H$ will be called \emph{base graphs} with
respect to their common CDC.

\section{Lifting and Guided Folding}
\label{sec:liftfold}

\subsection*{Lifting}

We introduce a construction analogous to the \emph{permutation voltage
graph} of Gross and Tucker~\cite{grosstucker01}, but with the permutations
restricted to those arising from TF-isomorphisms of the base graph. We retain
the established voltage-graph terminology and refer to the construction as a
\emph{lifting}.

Given a non-trivial TF-isomorphism $(\alpha,\beta)$ from a graph $G$ to a
graph $H$, the pair $(\alpha,\beta)$ \emph{lifts} $G$ to a digraph
$\vec{G}_{\alpha,\beta}$ with vertex set
$V(\vec{G}_{\alpha,\beta})=\alpha(V(G))\cup\beta(V(G))$ and arc set
determined by $(\alpha(u),\beta(v))\in A(\vec{G}_{\alpha,\beta})$
if and only if $(u,v)\in A(G)$.
The digraph $\vec{G}_{\alpha,\beta}$ has the same underlying structure as $\ADC(G)$,
with vertex labels encoded by $\alpha$ and~$\beta$. Where $\ADC(G)$ contains arcs
$((u,0),(v,1))$ and $((v,0),(u,1))$, the lift contains $(\alpha(u),\beta(v))$
and $(\alpha(v),\beta(u))$. Figure~\ref{fig:lift01} shows the lifting of
$G\cong C_3$.

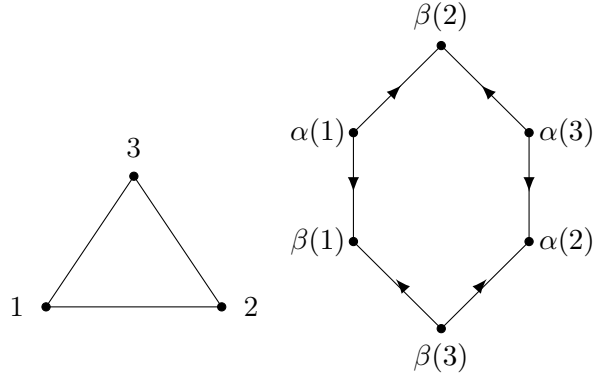
\begin{figure}[H]
\centering
\begin{tikzpicture}[node distance={15mm},thick,
    main/.style={draw,circle,inner sep=1.2pt,fill}]
\node[main](t1) at (0,0)        [label=left:{1}]{};
\node[main](t2) at (1.5,0)      [label=right:{2}]{};
\node[main](t3) at (0.75,1.299) [label=above:{3}]{};
\draw(t1)--(t2);\draw(t2)--(t3);\draw(t3)--(t1);
\node[main](1) at (4.9,2.45)     [label=above:{$\beta(2)$}]{};
\node[main](2) [below right of=1][label=right:{$\alpha(3)$}]{};
\node[main](3) [below of=2]      [label=right:{$\beta(1)$}]{};
\node[main](4) [below left of=3] [label=below:{$\alpha(2)$}]{};
\node[main](5) [above left of=4] [label=left:{$\beta(3)$}]{};
\node[main](6) [above of=5]      [label=left:{$\alpha(1)$}]{};
\draw[directed,thick](2)--(1);\draw[directed,thick](2)--(3);
\draw[directed,thick](4)--(3);\draw[directed,thick](4)--(5);
\draw[directed,thick](6)--(5);\draw[directed,thick](6)--(1);
\end{tikzpicture}
\caption{The graph $G\cong C_3$ (left) and its lift
         $\vec{G}_{\alpha,\beta}\cong\ADC(G)$ (right).}
\label{fig:lift01}
\end{figure}

If $(\alpha,\beta)$ is a TF-isomorphism from $H$ to $K$, then
$(\alpha^{-1},\beta^{-1})$ is a TF-isomorphism from $K$ to $H$ and
$\vec{H}_{\alpha,\beta}\cong\vec{K}_{\alpha^{-1},\beta^{-1}}$; in
particular, the lifts of TF-isomorphic base graphs are isomorphic, and
replacing the arcs of any lift with undirected edges recovers the common CDC.

The following two observations will be used in the constructions of Section~\ref{sec:constructions}.

\begin{Prop}\label{prop:bipartitelift}
Let $G$ be connected and let $(\alpha,\beta)$ be a TF-isomorphism from $G$ to some graph $H$. Then the lift $\vec{G}_{\alpha,\beta}$ is disconnected if and only if $G$ is bipartite.
\end{Prop}

\begin{proof}
If $G$ is bipartite with colour classes $A$ and $B$, every arc of
$\vec{G}_{\alpha,\beta}$ runs from $\alpha(A)$ to $\beta(B)$ or from
$\alpha(B)$ to $\beta(A)$, so the lift splits into two components, each a
reversal of the other.

Conversely, if $\vec{G}_{\alpha,\beta}$ is disconnected, there exists
$u\in V(G)$ for which no \textbf{A}-trail joins $\alpha(u)$ to $\beta(u)$,
implying that $G$ has no odd closed trail through $u$. Since every connected
non-bipartite graph admits an odd closed trail through each of its
vertices~\cite{bondy}, $G$ must be bipartite.
\end{proof}

\begin{Prop}\label{prop:disconnectedbase}
If $G$ is disconnected, then $\vec{G}_{\alpha,\beta}$ is disconnected
for every TF-isomorphism $(\alpha,\beta)$ from $G$ to a graph $H$.
\end{Prop}

\begin{proof}
Each arc of $\vec{G}_{\alpha,\beta}$ connects $\alpha(u)$ to $\beta(v)$
for an arc $(u,v)$ of $G$. If $x$ and $y$ lie in different components
of $G$, no \textbf{A}-trail joins $\alpha(x)$ to $\beta(y)$, so the lift
inherits the disconnection.
\end{proof}

\subsection*{Folding}

Recovering a graph from a lift is called \emph{folding} or
\emph{projection}~\cite{grosstucker01}. The idea is to associate each vertex
$\alpha(u)$ in one colour class of $\vec{G}_{\alpha,\beta}$ uniquely with a
vertex in the other class $\beta V(G)$ (writing $\alpha V(G)$ and
$\beta V(G)$ for $\alpha(V(G))$ and $\beta(V(G))$), then collapse each associated pair to
a single vertex, so that arcs and their reverses collapse to the same
undirected edge. The canonical identification $\alpha(u)\leftrightarrow\beta(u)$
recovers $G$. A folding is called \emph{trivial} if the resulting graph is
isomorphic to~$G$.

Formally, a \emph{stitching map} is a bijection
$\phi:\alpha V(G)\cup\beta V(G)\to\beta V(G)\cup\alpha V(G)$ between the
two colour classes satisfying
\[
    \phi(\alpha(u))=\beta(v) \iff \phi(\beta(v))=\alpha(u),
\]
and the associated \emph{stitcher} $\hat\phi$ is the surjection sending both
$\alpha(u)$ and $\phi(\alpha(u))$ to a common vertex $\hat{u}$ for every
$u\in V(G)$. Every stitching map induces a permutation of
$V(\CDC(G))$ that interchanges the colour classes, and the displayed
condition is precisely that this permutation is an involution.

A stitching map need not respect adjacency, and identifying two adjacent
vertices would create a loop, so we impose two further conditions.

\begin{Def}\label{def:guide}
A stitching map $\phi$ for $\vec{G}_{\alpha,\beta}$ is a \emph{guide} if
\begin{enumerate}
\item[(G1)] $\phi$ is an isomorphism from $\vec{G}_{\alpha,\beta}$ to its
  converse $\vec{G}_{\alpha,\beta}^{\opp}$; equivalently, the induced
  permutation of $V(\CDC(G))$ is an automorphism of $\CDC(G)$; and
\item[(G2)] no edge of $\CDC(G)$ joins a vertex to its own image, that is
  $\{x,\phi(x)\}\notin E(\CDC(G))$ for every $x\in V(\CDC(G))$.
\end{enumerate}
\end{Def}

Condition (G1) is what makes the folded graph TF-isomorphic to $G$, by
Lemma~\ref{prop:involution01}; condition (G2) is what keeps it a simple
graph, since the stitcher identifies $x$ with $\phi(x)$ and an edge between
them would fold to a loop. Together they hold exactly when the permutation
of $V(\CDC(G))$ induced by $\phi$ is a strongly switching involution in the
sense of Definition~\ref{def:switching}. They depend only on $\CDC(G)$ and
not on the lift, so we also speak of a guide for $\CDC(G)$.

The following example illustrates the two distinct outcomes.
Both conditions are satisfied in its two foldings.

\begin{Ex}\label{ex:c6fold}
Let $G\cong C_6$ with $V(G)=\{1,2,3,4,5,6\}$ and edges $\{i,i{+}1\bmod 6\}$.
Set $\alpha=(2\ 5)$ and $\beta=(3\ 6)(4\ 1)$. Each edge $\{u,v\}$ of $G$
contributes arcs $(\alpha(u),\beta(v))$ and $(\alpha(v),\beta(u))$ to the
lift, and the twelve arcs decompose into two \textbf{A}-connected circuits:
Circuit~1 has $\alpha$-class $\{\alpha(1),\alpha(3),\alpha(5)\}$ and
$\beta$-class $\{\beta(2),\beta(4),\beta(6)\}$; Circuit~2 has the
complementary classes.  We apply two guides, using the stitcher
$\hat\phi$ of the definition above.

\textit{Trivial folding.}
Set $\phi_1(\alpha(u))=\beta(u)$ for each $u\in V(G)$. The stitcher
$\hat\phi_1$ collapses $\alpha(u)$ and $\beta(u)$ to $\hat u$.
The edge $\{u,v\}$ of $G$ lifts to the two arcs
$(\alpha(u),\beta(v))$ and $(\alpha(v),\beta(u))$, which collapse to the
mutually reverse pair $(\hat u,\hat v)$ and $(\hat v,\hat u)$, that is, to
the single undirected edge $\{\hat u,\hat v\}$. Circuit~1
contains one arc of each such pair and Circuit~2 the other, so the two
circuits together give the six edges $\{\hat 1,\hat 2\}$,
$\{\hat 2,\hat 3\}$, $\{\hat 3,\hat 4\}$, $\{\hat 4,\hat 5\}$,
$\{\hat 5,\hat 6\}$, $\{\hat 6,\hat 1\}$ of $C_6$.  Hence $H\cong G$: a
trivial folding.

\textit{Non-trivial folding.}
Let $\bar u$ denote the antipodal of $u$ in $C_6$ (so $\bar 1=4$,
$\bar 2=5$, $\bar 3=6$, and vice versa), and set
$\phi_2(\alpha(u))=\beta(\bar u)$ for each $u$.  The stitcher $\hat\phi_2$
collapses $\alpha(u)$ and $\beta(\bar u)$ to $\hat u$.  Since $\beta(v)$
is collapsed with $\alpha(\bar v)$, it is sent to $\hat{\bar v}$,
so arc $(\alpha(u),\beta(v))$ becomes the arc
$(\hat u,\hat{\bar v})$, and $(\alpha(\bar v),\beta(\bar u))$ becomes its
reverse; together they are the edge $\{\hat u,\hat{\bar v}\}$. Here the two
arcs of a pair lie in the same circuit, so the six arcs of Circuit~1 collapse
in pairs to three distinct edges: $(\alpha(1),\beta(2))$ and
$(\alpha(5),\beta(4))$ both give $\{\hat 1,\hat 5\}$, and the full set
yields the triangle on $\hat 1$, $\hat 3$, $\hat 5$.  Circuit~2 similarly
yields the triangle on $\hat 2$, $\hat 4$, $\hat 6$.  Hence
$H\cong C_3\cup C_3\not\cong G$: a non-trivial folding, and $G$ and $H$
are TF-cousins.
\end{Ex}

The following three lemmas formalise the structural constraints on
non-trivial foldings; they build towards Theorem~\ref{prop:involution04},
the algebraic core of the framework.

\begin{Lem}\label{prop:involution01}
Let $G$ be a base graph, $(\alpha,\beta)$ a TF-isomorphism from $G$ to a
mixed graph $H$, $\phi$ a stitching map for
$\vec{G}_{\alpha,\beta}$, and $H'$ the
graph obtained by applying the stitcher $\hat\phi$. Then
$H'\cong^{\mathbf{TF}}G\cong^{\mathbf{TF}}H$ if and only if $\phi$ is an
isomorphism from $\vec{G}_{\alpha,\beta}$ to its converse
$\vec{G}_{\alpha,\beta}^{\opp}$.
\end{Lem}

\begin{proof}
Each pair identified by $\hat\phi$ contains exactly one vertex of
$\alpha V(G)$ and one of $\beta V(G)$, so $(\alpha(u),\beta(v))$ is the only arc of the lift
with image $(\hat\phi(\alpha(u)),\hat\phi(\beta(v)))$. That image is therefore
an arc of $H'$ if and
only if $(\alpha(u),\beta(v))\in A(\vec{G}_{\alpha,\beta})$, which holds if and
only if
$(\phi(\alpha(u)),\phi(\beta(v)))\in A(\vec{G}_{\alpha,\beta}^{\opp})$.
The second equivalence, for all $u$ and $v$, is precisely the
condition that $\phi$ should be an isomorphism from
$\vec{G}_{\alpha,\beta}$ to its converse: every arc of the lift has the form $(\alpha(u),\beta(v))$, and for a
pair of any other form both sides fail.
The claim follows from the chain of equivalences
\begin{align*}
(\phi(\alpha(u)),\phi(\beta(v)))\in A(\vec{G}_{\alpha,\beta}^{\opp})
&\iff(u,v)\in A(G)\\
&\iff(\alpha(u),\beta(v))\in A(H). \qedhere
\end{align*}
\end{proof}

\begin{Lem}\label{prop:involution02}
Let $\phi$ be a guide for $\vec{G}_{\alpha,\beta}$ and let $H'$ be the graph
obtained by folding along it, with vertex set $V_0$, in which $u$ and $v$ are
adjacent exactly when $\{u,\phi(v)\}\in E(\CDC(G))$. Let $\sigma$ denote the
involution
$(w,i)\mapsto(w,1-i)$ of $V(\CDC(G))$, which is the canonical identification
$\alpha(w)\leftrightarrow\beta(w)$. Then
\[
  (\id,\ \phi\sigma|_{V_0})
\]
is a TF-isomorphism from $G$ to $H'$, and hence $G$ and $H'$ are
TF-isomorphic. This pair need not be $(\alpha,\beta)$, the pair from which
the lift was formed.
\end{Lem}

\begin{proof}
Since $\phi$ interchanges the colour classes, for each $x\in V(\CDC(G))$
exactly one of $x$ and $\phi(x)$ lies in $V_0$, so $H'$ may indeed be taken
to have vertex set $V_0$. The map $\sigma$ is itself a guide: it satisfies
(G1) as an automorphism of $\CDC(G)$, and (G2) because an edge
$\{(w,0),(w,1)\}$ of $\CDC(G)$ would require a loop at $w$ in $G$. It folds
$\CDC(G)$ back to $G$.

Identify $V(G)$ with $V_0$. Since $\sigma$ maps $V_0$ to $V_1$ and $\phi$ maps
$V_1$ to $V_0$, the composite $\phi\sigma$ is a permutation of $V_0$. For
$u,v\in V_0$ we have $u\sim\phi\sigma(v)$ in $H'$ if and only if
$\{u,\phi(\phi\sigma(v))\}\in E(\CDC(G))$. Since $\phi^2=\id$, this condition
becomes $\{u,\sigma(v)\}=\{(u,0),(v,1)\}\in E(\CDC(G))$, which holds if and
only if $u\sim v$ in $G$.
\end{proof}

\begin{Lem}\label{prop:involution03}
Let $G$ be a base graph and $\phi$ a guide for $\CDC(G)$.
Then $\phi$ is a strongly switching involution in
$\Aut(\CDC(G))$.
\end{Lem}

\begin{proof}
A guide is a stitching map, so it induces an involution of
$V(\CDC(G))$ interchanging the colour classes. By condition (G1) of
Definition~\ref{def:guide} $\phi$ is an automorphism of $\CDC(G)$, hence a
switching involution, and by condition (G2) no edge joins a vertex to its own
image, so $\phi$ is strongly switching.
\end{proof}

\begin{Thm}\label{prop:involution04}
Let $G$ be connected and non-bipartite.
Let $\phi$ and $\phi'$ be guides for $\CDC(G)$, folding it into base graphs
$H$ and $H'$ respectively. Then $H\cong H'$ if and only if $\phi$ and
$\phi'$ are conjugate in $\Aut(\CDC(G))$. Consequently distinct conjugacy
classes of guides yield non-isomorphic base graphs sharing the same
canonical double cover.
\end{Thm}

\begin{proof}
As in Lemma~\ref{prop:involution02}, the fold of $\CDC(G)$ along
$\phi$ may be taken to have vertex set $V_0$, with $u$ and $v$ adjacent
exactly when $\{u,\phi(v)\}\in E(\CDC(G))$, a condition symmetric in $u$
and $v$ because $\phi$ is an automorphism with $\phi^2=\id$. Let $H_\phi$
denote this graph, so that $H=H_\phi$. Similarly,
$H'=H_{\phi'}$.

Suppose first that $\psi\phi=\phi'\psi$ for some $\psi\in\Aut(\CDC(G))$.
Since $G$ is connected and non-bipartite, $\CDC(G)$ is connected, so its
bipartition is unique and every automorphism of $\CDC(G)$ either preserves
the colour classes or interchanges them. Set $\rho=\psi$ if $\psi$ preserves
the colour classes, and $\rho=\phi'\psi$ otherwise. In the second case $\rho$
also conjugates $\phi$ to $\phi'$, since
$(\phi'\psi)\phi=\phi'(\psi\phi)=\phi'(\phi'\psi)=\phi'\rho$, and it preserves
the
colour classes because $\phi'$ and $\psi$ both interchange them. Either way,
$\rho\phi=\phi'\rho$ and $\rho$ restricts to a permutation of $V_0$. For
$u,v\in V_0$,
\begin{align*}
  \{u,\phi(v)\}\in E(\CDC(G))
  &\iff\{\rho(u),\rho\phi(v)\}\in E(\CDC(G))\\
  &\iff\{\rho(u),\phi'\rho(v)\}\in E(\CDC(G)),
\end{align*}
so the restriction of $\rho$ to $V_0$ is an isomorphism from $H_\phi$ to
$H_{\phi'}$.

Conversely, let $g:H_\phi\to H_{\phi'}$ be an isomorphism and define
$\rho$ on $V(\CDC(G))$ by $\rho(x)=g(x)$ for $x\in V_0$ and
$\rho(x)=\phi'g\phi(x)$ for $x\in V_1$. Then $\rho$ is a permutation of
$V(\CDC(G))$ preserving the colour classes. Let $\{u,y\}$ be a pair with
$u\in V_0$ and $y\in V_1$. Since $\phi(y)\in V_0$ and $\phi^2=\id$, we have
$\{u,y\}\in E(\CDC(G))$ if and only if $u$ and $\phi(y)$ are adjacent in
$H_\phi$, hence if and only if $g(u)$ and $g\phi(y)$ are adjacent in
$H_{\phi'}$, that is, if and only if
$\{g(u),\phi'g\phi(y)\}=\{\rho(u),\rho(y)\}\in E(\CDC(G))$.
Every edge of $\CDC(G)$ joins $V_0$ to $V_1$, so this covers every
edge. Thus $\rho\in\Aut(\CDC(G))$. Finally $\rho\phi=\phi'\rho$, since for
every $x\in V_0$ both sides equal $\phi'g(x)$ and for every $x\in V_1$ both
equal $g\phi(x)$. The first case uses $\phi^2=\id$ and the second
$\phi'^2=\id$.
Hence $\phi$ and $\phi'$ are conjugate.
\end{proof}

\begin{Rem}\label{rem:pacco}
Theorem~\ref{prop:involution04} is Theorem~3.5 of Pacco and
Scapellato~\cite{pacco} expressed in the language of guides, and the proof
above follows theirs. They work with digraphs and attribute the
undirected case to Porcu~\cite{porcu}. In that setting neither condition of
Definition~\ref{def:guide} is an extra requirement: a guide arises as an isomorphism
between alternating double covers, which is (G1), and it is strongly
switching exactly when the folded graph is loopless, which is (G2). So the
correspondence
restricts to one between conjugacy classes of strongly switching
involutions and isomorphism classes of loopless graphs sharing a given CDC.
What the present framework adds is not the correspondence itself but an
explicit construction: each guide is a folding of the lift,
so that a base graph is produced from an involution rather than only
counted.
\end{Rem}

\begin{Rem}\label{rem:bychawski}
Bychawski~\cite{bychawski} develops a general theory of the automorphism
groups of canonical double covers. He shows, among other things, that the
canonical double cover of an asymmetric graph has abelian automorphism group
of odd order, and that every semidirect product $H\rtimes\ZZ_2$ arises as
$\Aut(\CDC(G))$ for some connected, non-bipartite, vertex-determining $G$.
He also constructs graphs realising any prescribed number and type of
TF-isomorphic companions. His Theorem~8.6 gives a bijection
between the conjugacy classes of strongly switching elements of
$\Aut(\CDC(G))$ and the non-isomorphic graphs TF-isomorphic to
$G$, for $G$ connected, non-bipartite and vertex-determining
(Bychawski's \emph{reduced} graphs), which is the setting adopted here. That bijection coincides with
Theorem~\ref{prop:involution04}, which he attributes to Pacco and
Scapellato. He goes further by showing that the cardinality of
each such conjugacy class equals the index of instability of the
corresponding base graph. The Desargues graph illustrates both statements.
Its automorphism group has order $240$ and contains eleven strongly
switching involutions, which fall into two conjugacy classes, one of size
$1$ and the other of size~$10$. The class of size~$1$ corresponds to the Petersen graph, which is
stable. The class of size~$10$ corresponds to its TF-cousin, the other cubic
graph on ten vertices also having the Desargues graph as canonical double
cover, whose index of instability is therefore~$10$. That pair, first identified by Porcu~\cite{porcu}, appears
below as $(\CG(1),\CG'(1))$ in Section~\ref{sec:claw}. Bychawski's
Corollary~3.5 records the stability of the Petersen graph. That it
nonetheless has a TF-cousin shows that having one does not by itself make a
graph unstable.

\end{Rem}

\begin{Rem}\label{rem:hammack}
The results of Imrich and Pisanski~\cite{imrichpisanski2008} and
Abay-Asmerom et al.~\cite{abayasmerom2010} on Kronecker covers of bipartite
graphs also fall within the present framework.
The key objects in both papers are involutions $\alpha\in\Aut(G)$ of a
bipartite graph $G$ that interchange the two colour classes, called
\emph{bipartition-reversing involutions} in~\cite{abayasmerom2010}.
Those that send no vertex to a neighbour of itself are the
\emph{polarities} of~\cite{imrichpisanski2008}, and they are precisely the
strongly switching involutions in $\Aut(G)$, with $G$ itself in the role of
the cover. Theorem~\ref{prop:involution04} then recovers the main
result of~\cite{imrichpisanski2008} (Proposition~3): the number of
non-isomorphic graphs $H$ with $G\cong H\times K_2$ equals the number of
conjugacy classes of polarities in $\Aut(G)$. Theorem~1
of~\cite{abayasmerom2010} runs over all bipartition-reversing involutions and
so also yields the graphs $H$ with loops, the polarities being those whose
fold is loopless.
\end{Rem}

\section{Constructing TF-Isomorphic Pairs and Unstable Graphs}
\label{sec:constructions}

In every known example of an unstable asymmetric graph or a pair of
non-isomorphic graphs with the same CDC, one can identify at least one pair
of odd circuits $C_k$ alongside a circuit $C_{2k}$ of twice the length. We
hereafter take $k$ to be an odd integer unless otherwise stated.

This pattern suggests the following strategy: for odd $k$, begin
with the disconnected graph $G_0\cong C_k\cup C_k$ paired with
$H_0\cong C_{2k}$, and progressively
extend this \emph{seed pair} by adding edges and vertices in a controlled
way. Since $G_0$ and $H_0$ share a CDC (they are TF-isomorphic), any extension
that respects the underlying TF-isomorphism produces TF-isomorphic
pairs, yielding either TF-cousins or unstable graphs depending
on which type of extension is applied. The lift of the seed pair
is disconnected, as expected from Proposition~\ref{prop:disconnectedbase},
since $G_0\cong C_k\cup C_k$ is itself disconnected.

We work throughout the case $k=3$ for concreteness.
Figure~\ref{fig:genesis01} shows a TF-isomorphism from $G_0\cong C_3\cup C_3$
to $H_0\cong C_6$, and Figure~\ref{fig:genesis02} shows the corresponding lift
$\vec{G_0}_{\alpha,\beta}$. The lift consists of two identical directed
6-circuits; any automorphism of one circuit that fixes arc orientations gives
a valid reassociation of the two components, and every such reassociation
corresponds to a strongly switching involution of $\CDC(G_0)$ in the same
conjugacy class in $\Aut(\CDC(G_0))$. Hence all such choices are equivalent and we may fix any one.

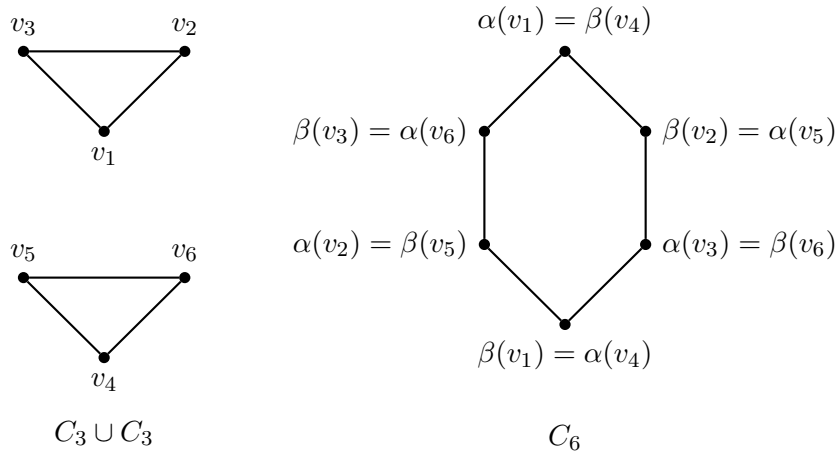
\begin{figure}[H]
\begin{center}
{
\begin{tikzpicture}[node distance={15mm},thick,above,
    main/.style={draw,circle,inner sep=1.2pt,fill},scale=0.4]
\node[main](1)[label=below:$v_1$]{};
\node[main](2)[above right of=1][label=$v_2$]{};
\node[main](3)[above left  of=1][label=$v_3$]{};
\node[main](4)[node distance=30mm,below of=1][label=below:$v_4$]{};
\node[main](5)[above left  of=4][label=$v_5$]{};
\node[main](6)[above right of=4][label=$v_6$]{};
\node[node distance=40mm,below of=1](20){$C_3\cup C_3$};
\draw(1)--(2);\draw(2)--(3);\draw(3)--(1);
\draw(4)--(5);\draw(5)--(6);\draw(6)--(4);
\node[main](7) [node distance=50mm,right of=2]
    [label={above:$\alpha(v_1)=\beta(v_4)$}]{};
\node[main](8) [below right of=7]
    [label={right:$\beta(v_2)=\alpha(v_5)$}]{};
\node[main](9) [below  of=8]
    [label={right:$\alpha(v_3)=\beta(v_6)$}]{};
\node[main](10)[below left of=9]
    [label={below:$\beta(v_1)=\alpha(v_4)$}]{};
\node[main](11)[above left of=10]
    [label={left:$\alpha(v_2)=\beta(v_5)$}]{};
\node[main](12)[above of=11]
    [label={left:$\beta(v_3)=\alpha(v_6)$}]{};
\draw(8)--(7);\draw(8)--(9);\draw(10)--(9);
\draw(10)--(11);\draw(12)--(11);\draw(12)--(7);
\node[node distance=15mm,below of=10]{$C_6$};
\end{tikzpicture}
}
\caption{A TF-isomorphism from $C_3\cup C_3$ to $C_6$.}
\label{fig:genesis01}
\end{center}
\end{figure}

\begin{figure}[H]
\begin{center}
{
\begin{tikzpicture}[node distance={15mm},thick,above,
    main/.style={draw,circle,inner sep=1.2pt,fill},scale=0.5]
\node[main](1) [label=above:$\alpha(v_1)$]{};
\node[main](2) [below right of=1][label=right:$\beta(v_2)$]{};
\node[main](3) [below  of=2]     [label=right:$\alpha(v_3)$]{};
\node[main](4) [below left of=3] [label=below:$\beta(v_1)$]{};
\node[main](5) [above left of=4] [label=left:$\alpha(v_2)$]{};
\node[main](6) [above of=5]      [label=left:$\beta(v_3)$]{};
\draw[directed,thick](1)--(2);\draw[directed,thick](3)--(2);
\draw[directed,thick](3)--(4);\draw[directed,thick](5)--(4);
\draw[directed,thick](5)--(6);\draw[directed,thick](1)--(6);
\node[main](7) [node distance=66mm,right of=1][label=above:$\beta(v_4)$]{};
\node[main](8) [below right of=7][label=right:$\alpha(v_5)$]{};
\node[main](9) [below  of=8]     [label=right:$\beta(v_6)$]{};
\node[main](10)[below left of=9] [label=below:$\alpha(v_4)$]{};
\node[main](11)[above left of=10][label=left:$\beta(v_5)$]{};
\node[main](12)[above of=11]     [label=left:$\alpha(v_6)$]{};
\draw[directed,thick](8)--(7);\draw[directed,thick](8)--(9);
\draw[directed,thick](10)--(9);\draw[directed,thick](10)--(11);
\draw[directed,thick](12)--(11);\draw[directed,thick](12)--(7);
\end{tikzpicture}
}
\caption{The lift $\vec{G_0}_{\alpha,\beta}$ corresponding to
         Figure~\ref{fig:genesis01}.}
\label{fig:genesis02}
\end{center}
\end{figure}
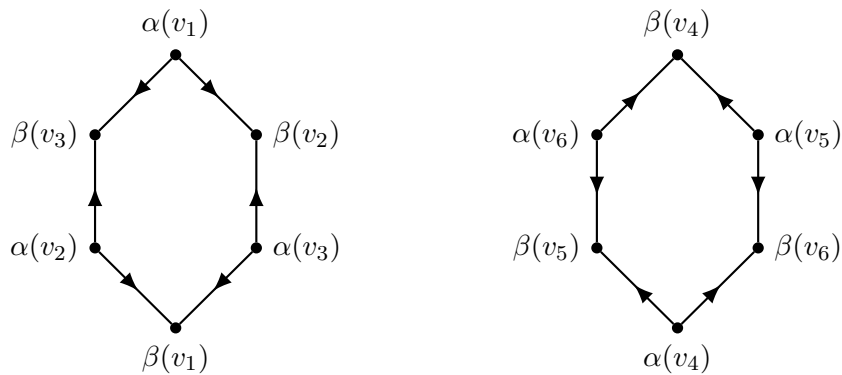

\subsection*{Extending a seed pair}

The following terminology captures two structural roles a vertex
can play relative to a TF-isomorphism.

\begin{Def}\label{def:pinentangle}
Let $(\alpha,\beta)$ be a TF-isomorphism. A vertex $x$ is a \emph{pin} if
$\alpha(x)=\beta(x)$. Two distinct vertices $x$, $y$ form an \emph{entangled
pair} if $\alpha(x)=\beta(y)$ and $\alpha(y)=\beta(x)$.
\end{Def}

In the seed pair above, the vertices $v_1,v_4$ form an entangled pair (as
do $v_2,v_5$ and $v_3,v_6$), since $\alpha(v_i)=\beta(v_{i+3})$ and
$\alpha(v_{i+3})=\beta(v_i)$ for $i=1,2,3$.

The seed pair may be extended in three ways, each producing a different
kind of result. The first two add edges or vertices directly
to it. The third, Proposition~\ref{prop:substitution}, starts from a
TF-isomorphic pair and produces new ones by substitution.

\begin{Con}[Entangled and split-image edges]\label{con:entangled}
Since $G_0$ is disconnected, at least one edge must be added to connect it.
Whenever an edge $\{v_i,v_j\}$ is added to $G_0$, its image under $(\alpha,\beta)$
(the arcs $(\alpha(v_i),\beta(v_j))$ and $(\alpha(v_j),\beta(v_i))$
of $H_0$) must be added simultaneously to preserve TF-isomorphism. The structure of the TF-isomorphism
then determines two types of edges:
\begin{itemize}
\item \emph{Entangled edges}: edges $\{v_i,v_j\}$ for which $v_i$, $v_j$ form
  an entangled pair. Such an edge cannot be added. Since
  $\alpha(v_i)=\beta(v_j)$, the arc $(v_i,v_j)$ maps to the loop
  $(\alpha(v_i),\alpha(v_i))$ at a single vertex of $H_0$, and $(v_j,v_i)$
  maps to a loop in the same way. An entangled pair is precisely a pair of
  vertices identified by the guide, so joining them by an edge creates the
  adjacency $u\sim\phi(u)$ excluded in Definition~\ref{def:switching}: the
  guide remains switching but ceases to be \emph{strongly} switching, and
  the folded graph acquires a loop. This is the correspondence between loops
  and switching involutions that fail to be strongly switching, recorded in
  Proposition~2.4 of~\cite{pacco} and attributed there to
  Porcu~\cite{porcu}.
\item \emph{Split-image edges}: edges whose image consists of two
  arcs that are not self-paired. Two such edges with
  the same image are \emph{complementary}. These must be added in complementary
  pairs, since a single such edge would add one edge to $G_0$ but
  two to $H_0$, while a complementary pair adds two to each. Doing so produces a pair of \emph{isomorphic} graphs, each
  admitting a non-trivial TF-automorphism and hence unstable
  by Theorem~\ref{thm:stabilitytfmain}.
\end{itemize}

Only split-image edges are therefore available, and adding
complementary pairs of them produces the unstable graphs in
Figure~\ref{fig:genesis04}. Since entangled edges are excluded, the seed
pair cannot be connected up at order~$6$ by adding edges alone, and indeed
no TF-cousin pair of connected graphs on six vertices exists
(Section~\ref{sec:conclusion}). The smallest connected TF-cousin pairs
occur at order~$7$ and are obtained by adding a pin
(Construction~\ref{con:pins}); there are exactly three of them, shown in
Figure~\ref{fig:sevenvertex}, the first being the pair of
Figure~\ref{fig:firstexample}.
\end{Con}

\begin{figure}[H]
\begin{center}
\scalebox{0.6}{
\begin{tikzpicture}[node distance={13mm},thick,above,
    main/.style={draw,circle,inner sep=2pt,fill}]
\node[main](1){};\node[main](2)[above right of=1]{};
\node[main](3)[above left of=1]{};\node[main](4)[node distance=26mm,below of=1]{};
\node[main](5)[above left of=4]{};\node[main](6)[above right of=4]{};
\draw(1)--(2);\draw(2)--(3);\draw(3)--(1);
\draw(4)--(5);\draw(5)--(6);\draw(6)--(4);\draw(3)--(5);\draw(2)--(6);
\node[main](7)[node distance=20mm,right of=2]{};
\node[main](8)[below right of=7]{};\node[main](9)[below of=8]{};
\node[main](10)[below left of=9]{};\node[main](11)[above left of=10]{};
\node[main](12)[above of=11]{};
\draw(8)--(7);\draw(8)--(9);\draw(10)--(9);
\draw(10)--(11);\draw(12)--(11);\draw(12)--(7);\draw(8)--(12);\draw(9)--(11);
\node[node distance=40mm,below of=1](25){\hspace{29mm}(a)};
\node[main](1b)[node distance=70mm,right of=1]{};
\node[main](2b)[above right of=1b]{};\node[main](3b)[above left of=1b]{};
\node[main](4b)[node distance=26mm,below of=1b]{};
\node[main](5b)[above left of=4b]{};\node[main](6b)[above right of=4b]{};
\draw(1b)--(2b);\draw(2b)--(3b);\draw(3b)--(1b);
\draw(4b)--(5b);\draw(5b)--(6b);\draw(6b)--(4b);
\draw(1b)--(6b);\draw(4b)--(3b);\draw(3b)--(5b);\draw(2b)--(6b);
\node[main](7b)[node distance=20mm,right of=2b]{};
\node[main](8b)[below right of=7b]{};\node[main](9b)[below of=8b]{};
\node[main](10b)[below left of=9b]{};\node[main](11b)[above left of=10b]{};
\node[main](12b)[above of=11b]{};
\draw(8b)--(7b);\draw(8b)--(9b);\draw(10b)--(9b);
\draw(10b)--(11b);\draw(12b)--(11b);\draw(12b)--(7b);
\draw(7b)--(9b);\draw(12b)--(10b);\draw(8b)--(12b);\draw(9b)--(11b);
\node[node distance=70mm,right of=25](26){(b)};
\node[main](1a)[node distance=140mm,right of=1]{};
\node[main](2a)[above right of=1a]{};\node[main](3a)[above left of=1a]{};
\node[main](4a)[node distance=26mm,below of=1a]{};
\node[main](5a)[above left of=4a]{};\node[main](6a)[above right of=4a]{};
\draw(1a)--(2a);\draw(2a)--(3a);\draw(3a)--(1a);
\draw(4a)--(5a);\draw(5a)--(6a);\draw(6a)--(4a);
\draw(1a)--(5a);\draw(2a)--(4a);\draw(1a)--(6a);
\draw(3a)--(4a);\draw(3a)--(5a);\draw(2a)--(6a);
\node[main](7a)[node distance=20mm,right of=2a]{};
\node[main](8a)[below right of=7a]{};\node[main](9a)[below of=8a]{};
\node[main](10a)[below left of=9a]{};\node[main](11a)[above left of=10a]{};
\node[main](12a)[above of=11a]{};
\draw(8a)--(7a);\draw(8a)--(9a);\draw(10a)--(9a);
\draw(10a)--(11a);\draw(12a)--(11a);\draw(12a)--(7a);
\draw(7a)--(11a);\draw(8a)--(10a);\draw(8a)--(12a);
\draw(9a)--(11a);\draw(7a)--(9a);\draw(12a)--(10a);
\node[node distance=85mm,right of=26]{(c)};
\end{tikzpicture}
}
\caption{Pairs of isomorphic unstable graphs from the same seed
         pair, by adding one, two, or three complementary pairs of
         split-image edges.}
\label{fig:genesis04}
\end{center}
\end{figure}

\begin{figure}[H]
\begin{center}
\scalebox{0.62}{
\begin{tikzpicture}[thick,main/.style={draw,circle,inner sep=2pt,fill}]
\foreach \lab/\dx in {a/0, b/6.4, c/12.8}{
  \begin{scope}[xshift=\dx cm]
    \foreach \i/\ang in {0/90,1/30,2/-30,3/-90,4/-150,5/150}
      \node[main](G\lab\i) at (\ang:1.5){};
    \foreach \i/\j in {0/1,1/2,2/3,3/4,4/5,5/0} \draw(G\lab\i)--(G\lab\j);
    \node[main](P\lab) at (0,0){};
    \begin{scope}[yshift=-4.9cm]
      \begin{scope}[xshift=-2cm]
        \foreach \i/\ang in {0/90,1/200,2/340} \node[main](A\lab\i) at (\ang:0.85){};
        \draw(A\lab0)--(A\lab1)--(A\lab2)--(A\lab0);
      \end{scope}
      \begin{scope}[xshift=2cm]
        \foreach \i/\ang in {0/90,1/340,2/200} \node[main](B\lab\i) at (\ang:0.85){};
        \draw(B\lab0)--(B\lab1)--(B\lab2)--(B\lab0);
      \end{scope}
      \node[main](Q\lab) at (0,0){};
    \end{scope}
    \node at (0,-7.2){(\lab)};
  \end{scope}
}
\foreach \i in {0,3} \draw(Pa)--(Ga\i);
\draw(Qa)--(Aa0);\draw(Qa)--(Ba0);
\foreach \i in {0,3,1,4} \draw(Pb)--(Gb\i);
\draw(Qb)--(Ab0);\draw(Qb)--(Bb0);
\draw(Qb) to[bend right=12](Ab1); \draw(Qb) to[bend left=12](Bb1);
\foreach \i in {0,1,2,3,4,5} \draw(Pc)--(Gc\i);
\draw(Qc)--(Ac0);\draw(Qc)--(Bc0);
\draw(Qc) to[bend right=12](Ac1); \draw(Qc) to[bend left=12](Bc1);
\draw(Qc) to[bend left=12](Ac2);  \draw(Qc) to[bend right=12](Bc2);
\end{tikzpicture}
}
\caption{The three pairs of TF-cousins on seven vertices, obtained from the
         seed pair by adding a pin joined to one, two, or three entangled
         pairs. These are the smallest connected TF-cousin pairs, and up to
         isomorphism they are the only ones on seven vertices. In each pair
         the upper graph contains $C_6$ and the lower contains
         $C_3\cup C_3$.}
\label{fig:sevenvertex}
\end{center}
\end{figure}

\begin{Con}[Adding vertices via pins]\label{con:pins}
The construction extends naturally to larger vertex sets. Adding an isolated
vertex $x$ to $G_0$ and an isolated vertex $y$ to $H_0$, and extending
$(\alpha,\beta)$ by $\alpha(x)=\beta(x)=y$, makes $x$ a pin. Connecting
$x$ to an entangled pair $v_i$, $v_j$ in $G_0$ by edges $\{v_i,x\}$ and
$\{v_j,x\}$, and adding the corresponding images to $H_0$, produces a
TF-isomorphic pair of order~$7$. The argument extends to pins of arbitrary
degree and to chains of pins, giving TF-isomorphic pairs of
every order from~$7$ on.
\end{Con}

The following result produces new TF-isomorphic pairs from old ones by a
different mechanism: replacing entangled vertices by odd circuits.
In that result the labelling of the neighbours of the two entangled vertices
is not free: it is determined by the TF-isomorphism, as the next lemma
shows.

\begin{Lem}\label{lem:tau}
Let $(\alpha,\beta)$ be a TF-isomorphism from a graph $G$ to a graph $H$, and
let $u$, $v$ be an entangled pair in $G$. Then $u$ and $v$ are non-adjacent,
and $\tau=\beta^{-1}\alpha$ interchanges them and restricts to a bijection
from $N(u)$ to $N(v)$.
\end{Lem}

\begin{proof}
If $u$ and $v$ were adjacent, the arc $(u,v)$ of $G$ would map to a loop in
$H$ at $\alpha(u)=\beta(v)$, and $H$ has no loops. Thus $u$ and $v$ are
non-adjacent.
Since $\alpha(u)=\beta(v)$ we have $\tau(u)=v$, and similarly $\tau(v)=u$. For
any $w\in V(G)$ we have $\beta(\tau(w))=\alpha(w)$, so
$\{v,\tau(w)\}\in E(G)$ if and only if
$\{\alpha(v),\beta(\tau(w))\}=\{\beta(u),\alpha(w)\}\in E(H)$, which holds if
and only if $\{w,u\}\in E(G)$. Hence $\tau(w)\in N(v)$ if and only if
$w\in N(u)$, and since $\tau$ is a permutation of $V(G)$ it restricts to a
bijection from $N(u)$ to $N(v)$.
\end{proof}

\begin{Prop}\label{prop:substitution}
Let $(\alpha,\beta)$ be a non-trivial TF-isomorphism from a graph $G$ to a
graph $H$, and let $u$, $v$ be an entangled pair in $G$, each of odd
degree~$k$. Let $\tau$ denote $\beta^{-1}\alpha$ and assume that
$\tau^2(w)=w$ for every $w\in N(u)$; equivalently, every neighbour of $u$ is a
pin or belongs to an entangled pair. Label the neighbours of $u$ as
$1,\dots,k$ and those of $v$ as $k+i=\tau(i)$. By Lemma~\ref{lem:tau} this
labelling lists every neighbour of $v$ exactly once. Form $G'$ by replacing
$u$ with an odd circuit $C_k$ on vertices $u_1,\dots,u_k$ and $v$ with a
second such circuit on vertices $u_{k+1},\dots,u_{2k}$. Attach $u_i$ to
neighbour $i$ and $u_{k+i}$ to neighbour $k+i$ for $i=1,\dots,k$. Form $H'$
from $H$ by deleting
the two vertices $\alpha(u)=\beta(v)$ and $\alpha(v)=\beta(u)$,
and introducing $2k$ new vertices, the
images of $u_1,\dots,u_{2k}$ under the extension defined below. Each of them
carries two names, exactly one of which has index at most $k$, so the list
$\alpha(u_1),\dots,\alpha(u_k)$, $\beta(u_1),\dots,\beta(u_k)$ names every new
vertex exactly once. Join
$\alpha(u_i)$ to $\beta(u_{i+1})$ and $\beta(u_i)$ to $\alpha(u_{i+1})$, the
indices read modulo $k$. As $k$ is odd, these edges form a single circuit
$C_{2k}$, on which the two kinds of name alternate. Join also
$\alpha(u_i)$ to $\beta(i)$ and $\beta(u_i)$ to $\alpha(i)$ for
$i=1,\dots,k$.

Extend $\alpha$ and $\beta$ from $V(G)\setminus\{u,v\}$ to the new vertices
$u_1,\dots,u_{2k}$ of $G'$ by
\[
  \alpha(u_i)=\beta(u_{i+k}),\qquad \beta(u_i)=\alpha(u_{i+k}),
\]
the indices being read modulo $2k$. The resulting pair is called the
\emph{extended pair}, again denoted by $(\alpha,\beta)$. Then the extended
pair is a TF-isomorphism from $G'$ to $H'$, and hence $G'$ and $H'$ are
TF-isomorphic.
\end{Prop}

\begin{proof}
Let $(\alpha,\beta)$ be the extended pair defined in the
statement. Every edge of $G'$ either lies in $G-\{u,v\}$ or is incident with
a new vertex. In the first case the defining condition already holds, since $\alpha$
and $\beta$ are unchanged there, so only the second case needs checking. Such
edges are the circuit edges
$\{u_i,u_{i+1}\}$ and $\{u_{k+i},u_{k+i+1}\}$ and the attachment edges
$\{u_i,i\}$ and $\{u_{k+i},k+i\}$.

Under $(\alpha,\beta)$ the images of the circuit edges are exactly the edges
of the circuit $C_{2k}$ of $H'$, since $\alpha(u_i)=\beta(u_{i+k})$. The
attachment edge $\{u_i,i\}$ gives the two edges $\{\alpha(u_i),\beta(i)\}$ and
$\{\alpha(i),\beta(u_i)\}$, and
$\{u_{k+i},k+i\}$ gives $\{\alpha(u_{k+i}),\beta(k+i)\}$ and
$\{\alpha(k+i),\beta(u_{k+i})\}$. Here $\alpha(u_{k+i})=\beta(u_i)$ and
$\beta(u_{k+i})=\alpha(u_i)$, while $\beta(k+i)=\beta(\tau(i))=\alpha(i)$
and, by the hypothesis $\tau^2(w)=w$, $\alpha(k+i)=\alpha(\tau(i))=\beta(i)$.

The four images therefore coincide in pairs, giving the two attachment edges
of $H'$ named in the statement, and each new vertex of $H'$ has exactly one
neighbour outside the circuit. By construction, $H'$ has no new edges beyond
the images just listed, so each arises from an edge of $G'$, which gives the
converse implication.
The antipodal pairing $\alpha(u_i)=\beta(u_{i+k})$ is precisely the
TF-isomorphism of the seed pair $(C_k\cup C_k,\,C_{2k})$.
\end{proof}

\begin{Rem}\label{rem:substitution-noniso}
The conclusion cannot be strengthened to non-isomorphism without further
hypotheses, and we know of none that suffices in general. An example on seven
vertices shows this. Let $G=H$ be the graph on $v_0,\dots,v_6$ with edges
\[
  v_0v_4,\ v_0v_5,\ v_0v_6,\ v_1v_4,\ v_1v_5,\ v_1v_6,\ v_2v_6,\ v_3v_6,\
  v_4v_6,\ v_5v_6.
\]
Here $(\alpha,\beta)$ with $\alpha=(v_0\,v_5)(v_1\,v_4)$ and
$\beta=(v_0\,v_4)(v_1\,v_5)$ is a non-trivial TF-automorphism of $G$, the
vertices $v_0$ and $v_1$ form an entangled pair, each of degree $3$, and
$\tau^2$ fixes $N(v_0)$ pointwise. The construction then yields graphs $G'$
and $H'$ on eleven vertices that are TF-isomorphic \emph{and} isomorphic. A
search over the connected non-bipartite graphs on seven and eight vertices in
McKay's catalogues~\cite{mckaydata,mckay} shows that such cases arise whether
or not the two vertices of the entangled pair have common neighbours. Among
the TF-cousin pairs on nine vertices they also arise when the construction
starts from two non-isomorphic graphs. What the construction guarantees is
therefore the TF-isomorphism; whether $G'$ and $H'$ are also
non-isomorphic has to be checked case by case.
\end{Rem}

A pair satisfying the hypotheses
of Proposition~\ref{prop:substitution} is readily available: any graph
$G$ in which two vertices share the same neighbourhood, of odd size
$k$, serves as a starting point with $H=G$. Increasing
$k$ then produces an infinite family of TF-isomorphic
pairs. Non-isomorphism does not follow from the construction alone, by
Remark~\ref{rem:substitution-noniso}, but it can be proved when the starting
graph is further restricted.

\begin{Cor}\label{cor:substitution-family}
Let $k\geq3$ be odd and let $G$ contain two vertices $u$, $v$ with
$N(u)=N(v)=W$ (so $G$ is not vertex-determining). Assume that $|W|=k$ and
that the vertices of degree $2$ in $G$ are exactly those of $W$. Take $H=G$,
$\alpha=(u\ v)$ and $\beta=\id$. Then the hypotheses of
Proposition~\ref{prop:substitution} are satisfied, and the resulting graphs
$G'$ and $H'$ are TF-cousins: TF-isomorphic and non-isomorphic. Taking
$G=K_{2,k}$ gives such a pair on $3k$ vertices.
\end{Cor}

\begin{proof}
Since $N(u)=N(v)$, a pair $\{x,y\}$ is an edge if and only if
$\{\alpha(x),y\}$ is one, so $(\alpha,\beta)$ is a TF-automorphism, and it is
non-trivial because $\alpha\neq\beta$. It has $\alpha(u)=v=\beta(v)$ and
$\alpha(v)=u=\beta(u)$, so $u$ and $v$ are entangled, each of degree $k$, and
non-adjacent, since a vertex is not its own neighbour. Here
$\tau=\beta^{-1}\alpha=(u\ v)$ fixes every vertex of $W$, so $\tau^2$ fixes
$N(u)$ pointwise, and in the labelling of
Proposition~\ref{prop:substitution} the names $i$ and $k+i$ denote the same
vertex of $W$. That proposition gives the TF-isomorphism.

For non-isomorphism, count degrees. In $G'$ the vertex $i\in W$ is adjacent to
$u_i$ and $u_{k+i}$ and to nothing else, so it has degree $2$; in $H'$, since
$\alpha$ and $\beta$ fix every vertex of $W$, it is adjacent to
$\alpha(u_i)$ and $\beta(u_i)$, again of degree $2$. Every new circuit vertex
has degree $3$, two neighbours on its circuit and one in $W$, and every
other vertex keeps the degree it had in $G$, which is not $2$ by hypothesis.
Hence in $G'$ and $H'$ too the vertices of degree $2$ are exactly those of
$W$. Any isomorphism $G'\to H'$ therefore maps $W$ onto $W$ and induces an
isomorphism $G'-W\to H'-W$. Outside their circuits the new vertices are adjacent only to
vertices of $W$, so in $G'-W$ they form two components, each a circuit of
length $k$, while in $H'-W$ they form one component, a circuit of length
$2k$. The remaining components are the same in both, since $H=G$ and the same
two vertices are deleted. Two copies of $C_k$
cannot be isomorphic to a single $C_{2k}$, so no such isomorphism exists.
\end{proof}

\begin{Rem}\label{rem:q5}
The constructions above are based on the seed pair $C_k\cup C_k$ and
$C_{2k}$, whose CDC is a union of even cycles. There exist
TF-cousin pairs with no such CDC.
As a concrete example, take the $5$-dimensional hypercube $Q_5$.
Among the six $H$ for which $Q_5\cong H\times K_2$
(enumerated in Theorem~2 and Figure~6 of~\cite{abayasmerom2010}), exactly three are
loopless: those corresponding to the conjugacy classes with parameters
$(3,0)$, $(3,1)$ and $(5,0)$ in the notation of~\cite{abayasmerom2010}.
Each is a $5$-regular graph on $16$ vertices, obtained by adding a perfect
matching to $Q_4$. A direct computation confirms the enumeration and shows
that the three have pairwise distinct adjacency spectra
(that is, the spectra of their adjacency matrices). The distinct
spectra make them pairwise non-isomorphic, and their common CDC makes them
TF-isomorphic by Theorem~\ref{thm:cdc}, so the three are mutually
TF-cousins. Their CDC is a hypercube rather than a union of even cycles.
They are nonetheless fully accounted for by the
conjugacy-class criterion of Theorem~\ref{prop:involution04}. This
illustrates that the lifting-and-folding framework is strictly more general
than the seed-pair method.
\end{Rem}

\section{Claw Graphs}
\label{sec:claw}

The constructions of Section~\ref{sec:constructions} establish that TF-cousin
pairs and unstable graphs exist at every order from~$6$ on, but the resulting graphs are
not easy to describe in closed form: they depend on the particular choices made
at each step. This section presents an explicit infinite family where the
TF-cousin question reduces to a single arithmetic condition on the parameter~$n$.

The building block is $K_{1,3}$, the graph consisting of a centre vertex
joined to three leaves, commonly called a claw. The circuit $C_{6n}$ has $3n$ antipodal pairs,
which are grouped into $n$ triples; we attach one claw to each triple.
Each leaf connects to exactly one antipodal pair and so will act as a pin in the
sense of Definition~\ref{def:pinentangle}. The result, denoted by $\CG(n)$,
is
a connected cubic graph on $10n$ vertices. Its companion $\CG'(n)$ is
obtained by replacing the single circuit $C_{6n}$ with two disjoint copies of
$C_{3n}$, leaving all claw connections unchanged.

We show that $\CG(n)$ and $\CG'(n)$ are TF-cousins if and only if $n$ is
odd. For $n=1$, $\CG(1)$ is the Petersen graph and $\CG'(1)$ is its
TF-cousin, a pair first identified by Porcu~\cite{porcu}. Their
common CDC is the Desargues graph
$\mathrm{GP}(10,3)$~\cite{imrichpisanski2008}, which Krnc and
Pisanski~\cite{krncpisanski2019} showed to be the only generalised Petersen
graph that is the canonical double cover of two non-isomorphic graphs.

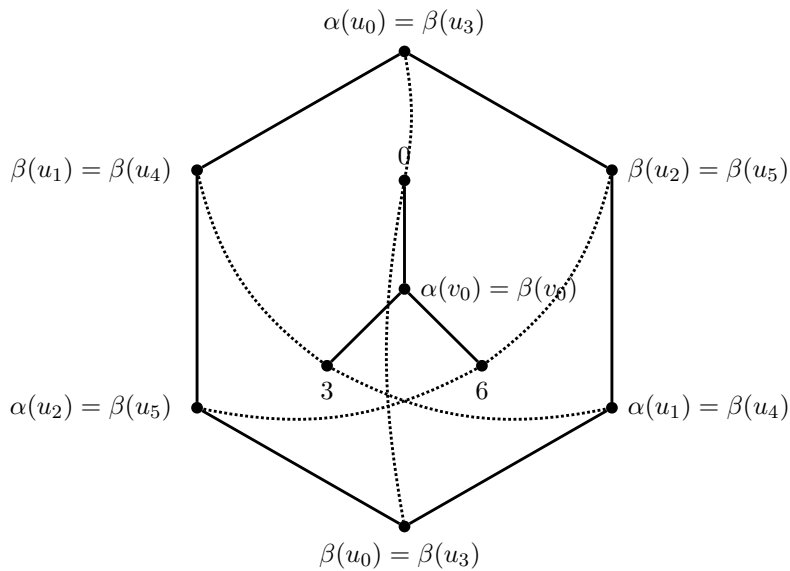
\begin{figure}[H]
\begin{center}
\scalebox{0.9}{
\begin{tikzpicture}[scale=0.7,node distance={16mm},very thick,above,
    main/.style={draw,circle,inner sep=1.3pt,fill}]
\node[main](u0) at (90:5)    [label={above:$u_{0}$}]{};
\node[main](u1) at (90+60:5) [label={left:$u_{1}$\ }]{};
\node[main](u2) at (90+120:5)[label={left:$u_{2}$\ }]{};
\node[main](u3) at (90+180:5)[label={below:$u_{3}$\ }]{};
\node[main](u4) at (90+240:5)[label={right:$u_{4}$\ }]{};
\node[main](u5) at (90+300:5)[label={right:$u_{5}$\ }]{};
\node[main](v0) at (0,0)[label={right:$c_0$}]{};
\node[main](v1)[above of=v0][label=left:$\ell_{0,0}$]{};
\node[main](v2)[below left  of=v0][label=below:$\ell_{0,1}$]{};
\node[main](v3)[below right of=v0][label=below:$\ell_{0,2}$]{};
\draw(v0)--(v1);\draw(v0)--(v2);\draw(v0)--(v3);
\draw(u0)--(u1);\draw(u1)--(u2);\draw(u2)--(u3);
\draw(u3)--(u4);\draw(u4)--(u5);\draw(u5)--(u0);
\draw[densely dotted](v1) to [bend right=10](u0);
\draw[densely dotted](v1) to [bend right=10](u3);
\draw[densely dotted](v2) to [bend left=20] (u1);
\draw[densely dotted](v2) to [bend right=20](u4);
\draw[densely dotted](v3) to [bend left=20] (u2);
\draw[densely dotted](v3) to [bend right=20](u5);
\end{tikzpicture}
}
\caption{The Petersen graph as a claw graph ($r=3$, $n=1$).
Vertices are named as in Definition~\ref{def:clawgraph}.}
\label{fig:clawgraph01}
\end{center}
\end{figure}

Writing $K_{1,r}$ for a claw with $r$ leaves, we fix $r=3$ throughout this
section so that the resulting graphs are cubic:
claw centres and leaves both have degree~$3$, as does every circuit
vertex. The general case $r\geq 2$ is addressed
in Remark~\ref{rem:generalk}.

\begin{Def}\label{def:clawgraph}
For a positive integer $n$, the \emph{claw graph} $\CG(n)$ is constructed
as follows. Take a circuit $C_{6n}$ with vertices
$u_0,\dots,u_{6n-1}$ and attach $n$ disjoint copies of $K_{1,3}$, one for
each residue $0\leq i<n$. The $i$-th copy has centre $c_i$ and leaves
$\ell_{i,0},\ell_{i,1},\ell_{i,2}$; leaf $\ell_{i,j}$ is joined to the
pair $\{u_{i+jn},\, u_{i+jn+3n}\}$ of antipodal circuit vertices
(indices modulo~$6n$). The result is a cubic graph on $10n$ vertices.

Its \emph{companion} $\CG'(n)$ is obtained from $\CG(n)$ by replacing the
single circuit $C_{6n}$ with two disjoint copies of $C_{3n}$, retaining the
vertices $u_0,\dots,u_{6n-1}$ and making $u_0,\dots,u_{3n-1}$ the first
copy and $u_{3n},\dots,u_{6n-1}$ the second. All claw connections are
unchanged.
\end{Def}

The graphs $\CG(n)$ and $\CG'(n)$ share the same vertex set and the same
claw edges; they differ only in the circuit edges. For $n=1$,
$\CG(1)$ is the Petersen graph (Figure~\ref{fig:clawgraph01}); for $n=3$,
$\CG(3)$ and its companion $\CG'(3)$ are shown in
Figures~\ref{fig:clawgraph02} and~\ref{fig:clawgraph02b}.

\begin{Prop}\label{prop:clawgeneralisation}
$\CG(n)$ and $\CG'(n)$ are TF-cousins if and only if $n$ is odd.
\end{Prop}

\begin{proof}
We prove the three parts in logical order: non-isomorphism first (a
pure graph-theoretic argument independent of TF-isomorphism theory), then
the $n$-even obstruction using it, then the $n$-odd construction.

\textit{Step~1: Non-isomorphism for all $n\geq 1$.}
The argument that follows deletes the claw centres from both graphs
and compares what is left. For that we need to know that an isomorphism
carries centres to centres. The construction names the centres, but an
isomorphism preserves only adjacency, so the centres have to be characterised by
adjacency alone. The $6$-circuits do that.
We claim that for $n\geq3$, a vertex of $\CG(n)$ or $\CG'(n)$ lies on no
$6$-circuit if and only if it is a claw centre.

First, every circuit vertex and every leaf lies on a $6$-circuit.
In either graph the shift $u_m\mapsto u_{m+3n}$ (with indices read modulo
$6n$) carries circuit edges to circuit edges: in $\CG(n)$ it is a rotation of
$C_{6n}$, and in $\CG'(n)$ it exchanges the two copies of $C_{3n}$. So if
$u_a$ and $u_b$ are adjacent then so are $u_{a+3n}$ and $u_{b+3n}$. Let
$u_a$, $u_b$ be adjacent circuit vertices, $\ell$ the leaf attached to
$\{u_a,u_{a+3n}\}$ and $\ell'$ the leaf attached to $\{u_b,u_{b+3n}\}$. Then
\[
  u_a,\ u_b,\ \ell',\ u_{b+3n},\ u_{a+3n},\ \ell,\ u_a
\]
is a $6$-circuit: the four circuit vertices are distinct because $u_a$ and
$u_b$ are adjacent and therefore not antipodal, and $\ell\neq\ell'$ because
each leaf is attached to exactly one antipodal pair. Every circuit vertex has a
neighbour on its circuit and every leaf is attached to a circuit vertex, so
these circuits cover every circuit vertex and every leaf.

Second, no $6$-circuit passes through a centre. Suppose, for contradiction,
that one passes through the centre $c$. Since $c$ is adjacent only to its
leaves, such a circuit has the form $c,\ell,x,y,z,\ell'$ with $\ell,\ell'$ leaves of $c$. The neighbours
of a leaf are its centre and the two vertices of the antipodal pair it is
attached to. So $x$ and $z$ are circuit vertices, $x$ attached to $\ell$ and
$z$ to $\ell'$. The neighbours of a centre are leaves, so $y$, which is
adjacent to both $x$ and $z$, is a circuit vertex or a leaf. Let $d$ denote
the difference between the indices of $x$ and $z$. The contradiction comes
from two facts about $d$.

\emph{Fact 1, from the claw alone.} Two leaves of one claw are attached to
antipodal pairs whose indices differ by $mn$ modulo $3n$, with $m\in\{1,2\}$,
and the two attachments of a single leaf differ by $3n$. Hence $d$ is
congruent modulo $6n$ to one of
\[
  \pm n,\quad \pm 2n,\quad \pm 4n,\quad \pm 5n ,
\]
and in particular is divisible by $n$.

\emph{Fact 2, from the vertex $y$.} If $y$ is a circuit
vertex, then $x$ and $z$ lie at distance $2$ on a circuit, so $d\equiv\pm2$
modulo the length of that circuit. In $\CG(n)$ that length is $6n$ and
$d=\pm2$; in $\CG'(n)$ it is $3n$ and $d=\pm2$ or $\pm(3n-2)$. Since $3n$ is a
multiple of $n$, each of these values is congruent to $\pm2$ modulo $n$,
whereas Fact~1 gives $d\equiv0$ modulo $n$. Hence $0\equiv\pm2$
modulo $n$, which says that $n$ divides $2$, and that fails for $n\geq3$.
So $y$ cannot be a circuit vertex.
If $y$ is a leaf, then $x$ and $z$ are its two attachments, so $d=3n$. Here $3n$
is divisible by $n$, so divisibility alone does not rule the case out.
However, $3n$ is not one of the values displayed in Fact~1, so $y$
cannot be a leaf.

For $n\geq3$, then, no $6$-circuit passes through a centre, and the claim
follows. The condition
that $n$ divides $2$ leaves the cases $n=1$ and $n=2$ open, and the
characterisation does fail for them. In fact, every centre of $\CG(1)$,
$\CG'(1)$, $\CG(2)$ and $\CG'(2)$ lies on a $6$-circuit. This happens because the leaves of a claw attach at
intervals of $n$ around the circuit, whereas a $6$-circuit through the centre
needs two attachments at distance $2$ on the circuit. A multiple of $n$ is
$2$ only for $n\leq2$.

\medskip
\noindent Let $n\geq3$ and suppose $f:\CG(n)\to\CG'(n)$ is an isomorphism.
Since $f$ preserves the property of lying on no $6$-circuit, it carries claw
centres to claw centres.
Let $G$ be either $\CG(n)$ or $\CG'(n)$, and let
$S = V(G)\setminus\{c_0,\dots,c_{n-1}\}$ denote the set of non-centre
vertices.
In the subgraph $G[S]$ induced by $S$, every circuit vertex $u_m$ retains
all three of its neighbours (two circuit, one leaf, all in $S$) and has
degree~$3$, while every leaf $\ell_{i,j}$ loses its edge to $c_i$ and has
degree~$2$. Hence the circuit vertices are precisely the degree-$3$ vertices
of $G[S]$. Removing the degree-$2$ vertices from $G[S]$ recovers the
subgraph induced by the $6n$ circuit vertices alone. In $\CG(n)$ this is
$C_{6n}$, which is connected. In $\CG'(n)$ it is $C_{3n}\cup C_{3n}$,
which has two components. The isomorphism $f$ would
commute with this procedure and so carry the (connected)
circuit subgraph of $\CG(n)$ to the (disconnected) circuit subgraph of
$\CG'(n)$, a contradiction. Hence $\CG(n)\not\cong\CG'(n)$ for
$n\geq3$.

The cases $n=1$ and $n=2$ are settled directly. For $n=1$ the graph $\CG'(1)$
contains a triangle, namely one of its two copies of $C_3$, while $\CG(1)$ is
the Petersen graph, which has girth~$5$. For $n=2$ a count of $6$-circuits
separates the two graphs: in $\CG'(2)$ every vertex lies on exactly six,
whereas in $\CG(2)$ each of the twelve circuit vertices lies on only five.

\begin{figure}[tp]
\begin{center}
\scalebox{0.80}{
\begin{tikzpicture}[node distance={15mm},very thick,above,
    main/.style={draw,circle,inner sep=1.5pt,fill}]
\node[main](u0)  at (90:6) [label={above:$u_{0}=\alpha(u_0)=\beta(u_9)$}]{};
\node[main](u1)  at (110:6) [label={left:$u_{1}=\alpha(u_{10})=\beta(u_1)$\ }]{};
\node[main](u2)  at (130:6) [label={left:$u_{2}=\alpha(u_2)=\beta(u_{11})\ $}]{};
\node[main](u3)  at (150:6) [label={left:$u_{3}=\alpha(u_{12})=\beta(u_3)$\ }]{};
\node[main](u4)  at (170:6) [label={left:$u_{4}=\alpha(u_4)=\beta(u_{13})\ $}]{};
\node[main](u5)  at (190:6) [label={left:$u_{5}=\alpha(u_{14})=\beta(u_5)$\ }]{};
\node[main](u6)  at (210:6) [label={left:$u_{6}=\alpha(u_6)=\beta(u_{15})\ $}]{};
\node[main](u7)  at (230:6) [label={left:$u_{7}=\alpha(u_{16})=\beta(u_7)$\ }]{};
\node[main](u8)  at (250:6) [label={left:$u_{8}=\alpha(u_8)=\beta(u_{17})\ $}]{};
\node[main](u9)  at (270:6) [label={below:$u_{9}=\alpha(u_9)=\beta(u_0)$}]{};
\node[main](u10) at (290:6) [label={right:$u_{10}=\alpha(u_1)=\beta(u_{10})$}]{};
\node[main](u11) at (310:6) [label={right:$u_{11}=\alpha(u_{11})=\beta(u_2)$}]{};
\node[main](u12) at (330:6) [label={right:$u_{12}=\alpha(u_3)=\beta(u_{12})$}]{};
\node[main](u13) at (350:6) [label={right:$u_{13}=\alpha(u_{13})=\beta(u_4)$}]{};
\node[main](u14) at (10:6) [label={right:$u_{14}=\alpha(u_5)=\beta(u_{14})$}]{};
\node[main](u15) at (30:6) [label={right:$u_{15}=\alpha(u_{15})=\beta(u_6)$}]{};
\node[main](u16) at (50:6) [label={right:$u_{16}=\alpha(u_7)=\beta(u_{16})$}]{};
\node[main](u17) at (70:6) [label={right:$u_{17}=\alpha(u_{17})=\beta(u_8)$}]{};
\foreach \i/\j in {0/1,1/2,2/3,3/4,4/5,5/6,6/7,7/8,8/9,
  9/10,10/11,11/12,12/13,13/14,14/15,15/16,16/17,17/0}
  \draw(u\i)--(u\j);
\node[main](v0) at (90:2.6)[label={above:$c_0$}]{};
\node[main](v1)[below of=v0][label={[xshift=2mm]below:$\ell_{0,0}$}]{};
\node[main](v2)[right of=v1][label={[xshift=2mm, yshift=1mm]below:$\ell_{0,1}$}]{};
\node[main](v3)[left  of=v1][label={[xshift=4.6mm, yshift=4.6mm]below:$\ell_{0,2}$}]{};
\draw(v0)--(v1);\draw(v0)--(v2);\draw(v0)--(v3);
\node[main](w0) at (90+240:2.6)[label={above:$c_1$}]{};
\node[main](w1)[below of=w0][label={[xshift=2mm]below:$\ell_{1,0}$}]{};
\node[main](w2)[right of=w1][label={[xshift=2mm]below:$\ell_{1,1}$}]{};
\node[main](w3)[left  of=w1][label={[xshift=2mm]below:$\ell_{1,2}$}]{};
\draw(w0)--(w1);\draw(w0)--(w2);\draw(w0)--(w3);
\node[main](x0) at (90+120:2.6)[label={above:$c_2$}]{};
\node[main](x1)[below of=x0][label={[xshift=-1.5mm]below:$\ell_{2,0}$}]{};
\node[main](x2)[right of=x1][label={[xshift=3.7mm, yshift=2mm]below:$\ell_{2,1}$}]{};
\node[main](x3)[left  of=x1][label=below:$\ell_{2,2}$]{};
\draw(x0)--(x1);\draw(x0)--(x2);\draw(x0)--(x3);
\draw[densely dotted](v1) to [bend right=80](u0);
\draw[densely dotted](v1) to [bend right=20](u9);
\draw[densely dotted](v2) to [bend left=60] (u3);
\draw[densely dotted](v2) to [bend left=60] (u12);
\draw[densely dotted](v3) to [bend right=60](u6);
\draw[densely dotted](v3) to [bend right=60](u15);
\draw[dotted](w1) to [bend right=50](u1);
\draw[dotted](w1) to [bend right=20](u10);
\draw[dotted](w2) to [bend left=75] (u4);
\draw[dotted](w2) to [bend right=20](u13);
\draw[dotted](w3) to [bend left=40] (u7);
\draw[dotted](w3) to [bend left=40] (u16);
\draw[loosely dotted](x1) to [bend right=40](u2);
\draw[loosely dotted](x1) to [bend right=50](u11);
\draw[loosely dotted](x2) to [bend right=75](u5);
\draw[loosely dotted](x2) to [bend left=20] (u14);
\draw[loosely dotted](x3) to [bend left=40] (u8);
\draw[loosely dotted](x3) to [bend left=40] (u17);
\end{tikzpicture}
}
\caption{The claw graph $\CG(3)$: the case $r=3$, $n=3$.
Claw vertices are named as in Definition~\ref{def:clawgraph} and are fixed by both $\alpha$ and $\beta$; a circuit vertex labelled $u_m=\alpha(u_a)=\beta(u_b)$ is $u_m$, the common image of $u_a$ under $\alpha$ and of $u_b$ under $\beta$, where $u_a$ and $u_b$ are vertices of $\CG'(3)$ as drawn in Figure~\ref{fig:clawgraph02b}.}
\label{fig:clawgraph02}
\end{center}
\end{figure}
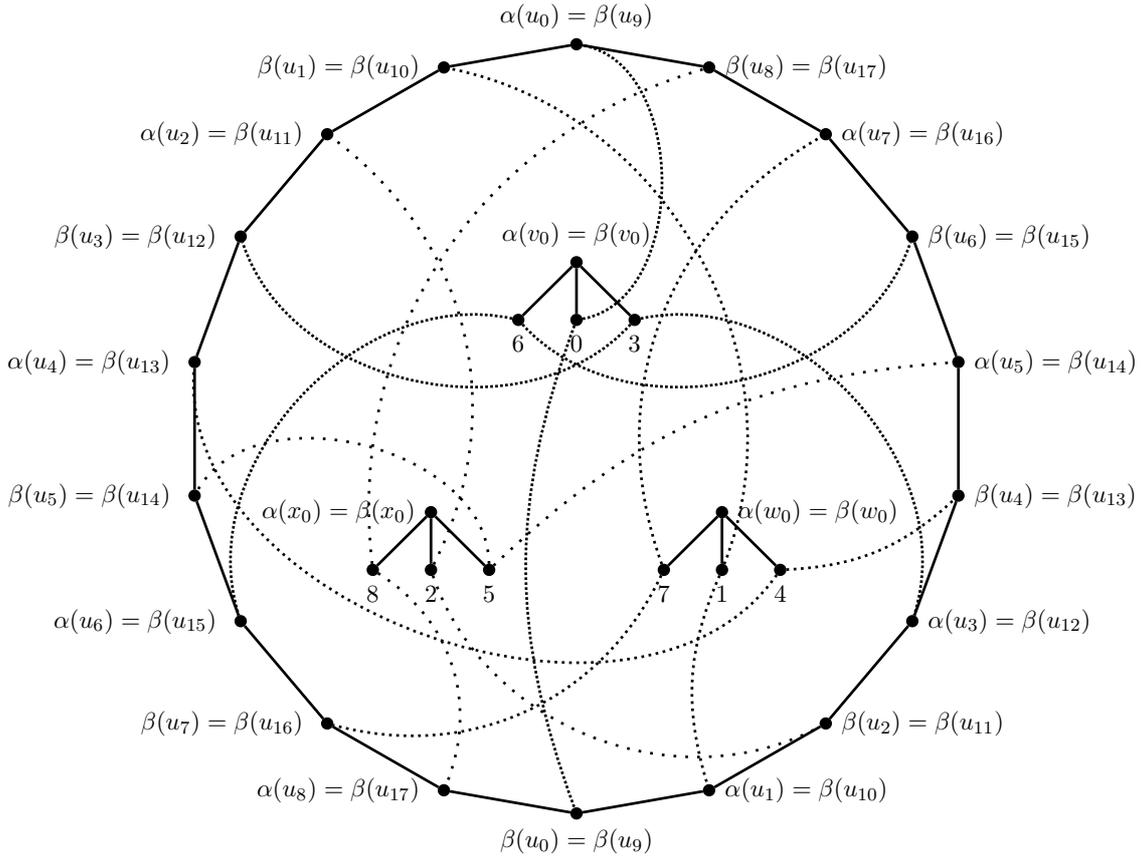

\begin{figure}[tp]
\begin{center}
\scalebox{0.85}{
\begin{tikzpicture}[scale=0.9, node distance={13mm},very thick,above,
    main/.style={draw,circle,inner sep=1.5pt,fill}]
\node[main](u0)  at (90+360*0/9:3) [label={above:$u_{0}$}]{};
\node[main](u1)  at (90+360*1/9:3) [label=left:$u_{1}$]{};
\node[main](u2)  at (90+360*2/9:3) [label=left:$u_{2}$]{};
\node[main](u3)  at (90+360*3/9:3) [label=left:$u_{3}$]{};
\node[main](u4)  at (90+360*4/9:3) [label=left:$u_{4}$]{};
\node[main](u5)  at (90+360*5/9:3) [label=right:$u_{5}$]{};
\node[main](u6)  at (90+360*6/9:3) [label=right:$u_{6}$]{};
\node[main](u7)  at (90+360*7/9:3) [label=right:$u_{7}$]{};
\node[main](u8)  at (90+360*8/9:3) [label=right:$u_{8}$]{};
\draw(u0)--(u1);\draw(u1)--(u2);\draw(u2)--(u3);
\draw(u3)--(u4);\draw(u4)--(u5);\draw(u5)--(u6);
\draw(u6)--(u7);\draw(u7)--(u8);\draw(u8)--(u0);
\tikzset{shift={(0,10)}}
\node[main](u9)  at (90+360*0/9:3) [label={above:$u_{9}$}]{};
\node[main](u10) at (90+360*1/9:3) [label=left:$u_{10}$]{};
\node[main](u11) at (90+360*2/9:3) [label=left:$u_{11}$]{};
\node[main](u12) at (90+360*3/9:3) [label=left:$u_{12}$]{};
\node[main](u13) at (90+360*4/9:3) [label={below:$u_{13}$}]{};
\node[main](u14) at (90+360*5/9:3) [label=right:$u_{14}$]{};
\node[main](u15) at (90+360*6/9:3) [label=right:$u_{15}$]{};
\node[main](u16) at (90+360*7/9:3) [label=right:$u_{16}$]{};
\node[main](u17) at (90+360*8/9:3) [label=right:$u_{17}$]{};
\draw(u9)--(u10);\draw(u10)--(u11);\draw(u11)--(u12);
\draw(u12)--(u13);\draw(u13)--(u14);\draw(u14)--(u15);
\draw(u15)--(u16);\draw(u16)--(u17);\draw(u17)--(u9);
\node[main](v0) at (7,1.5)[label=right:$c_0$]{};
\node[main](v1) at ([xshift=-10mm]v0)[label={above:$\ell_{0,0}$}]{};
\node[main](v2) at ([xshift=-13mm,yshift=13mm]v0)[label=right:$\ell_{0,1}$]{};
\node[main](v3) at ([xshift=-13mm,yshift=-13mm]v0)[label=right:$\ell_{0,2}$]{};
\draw(v0)--(v1);\draw(v0)--(v2);\draw(v0)--(v3);
\node[main](w0) at (-7,-5)[label=left:$c_1$]{};
\node[main](w1) at ([xshift=10mm]w0)[label={[xshift=0mm, yshift=-7.5mm]above:$\ell_{1,0}$}]{};
\node[main](w2) at ([xshift=13mm,yshift=-13mm]w0)[label=left:$\ell_{1,1}$]{};
\node[main](w3) at ([xshift=13mm,yshift=13mm]w0)[label=left:$\ell_{1,2}$]{};
\draw(w0)--(w1);\draw(w0)--(w2);\draw(w0)--(w3);
\node[main](x0) at (7,-11.5)[label=right:$c_2$]{};
\node[main](x1) at ([xshift=-10mm]x0)[label={[xshift=0mm, yshift=-7.5mm]above:$\ell_{2,0}$}]{};
\node[main](x2) at ([xshift=-13mm,yshift=13mm]x0)[label=right:$\ell_{2,1}$]{};
\node[main](x3) at ([xshift=-13mm,yshift=-13mm]x0)[label=right:$\ell_{2,2}$]{};
\draw(x0)--(x1);\draw(x0)--(x2);\draw(x0)--(x3);
\draw[densely dotted](v1) to [bend left=20]  (u0);
\draw[densely dotted](v1) to [bend right=10] (u9);
\draw[densely dotted](v2) to [bend right=19] (u3);
\draw[densely dotted](v2) to [bend right=15] (u12);
\draw[densely dotted](v3) to [bend left=23]  (u6);
\draw[densely dotted](v3) to [bend right=25] (u15);
\draw[dotted](w1) to [bend right=0]  (u1);
\draw[dotted](w1) to [bend left=26]  (u10);
\draw[dotted](w2) to [bend right=35] (u4);
\draw[dotted](w2) to [bend right=0]  (u13);
\draw[dotted](w3) to [bend left=0]   (u7);
\draw[dotted](w3) to [bend left=0]   (u16);
\draw[loosely dotted](x1) to [bend right=0]  (u2);
\draw[loosely dotted](x1) to [bend right=10] (u11);
\draw[loosely dotted](x2) to [bend left=10]  (u5);
\draw[loosely dotted](x2) to [bend left=0]   (u14);
\draw[loosely dotted](x3) to [bend right=16]   (u8);
\draw[loosely dotted](x3) to [bend left=0]   (u17);
\end{tikzpicture}
}
\caption{The companion $\CG'(3)$, which is TF-isomorphic but non-isomorphic to $\CG(3)$.
Vertices are named as in Definition~\ref{def:clawgraph}. The map of Step~3 is defined on these vertices; Figure~\ref{fig:clawgraph02} shows their images.}
\label{fig:clawgraph02b}
\end{center}
\end{figure}
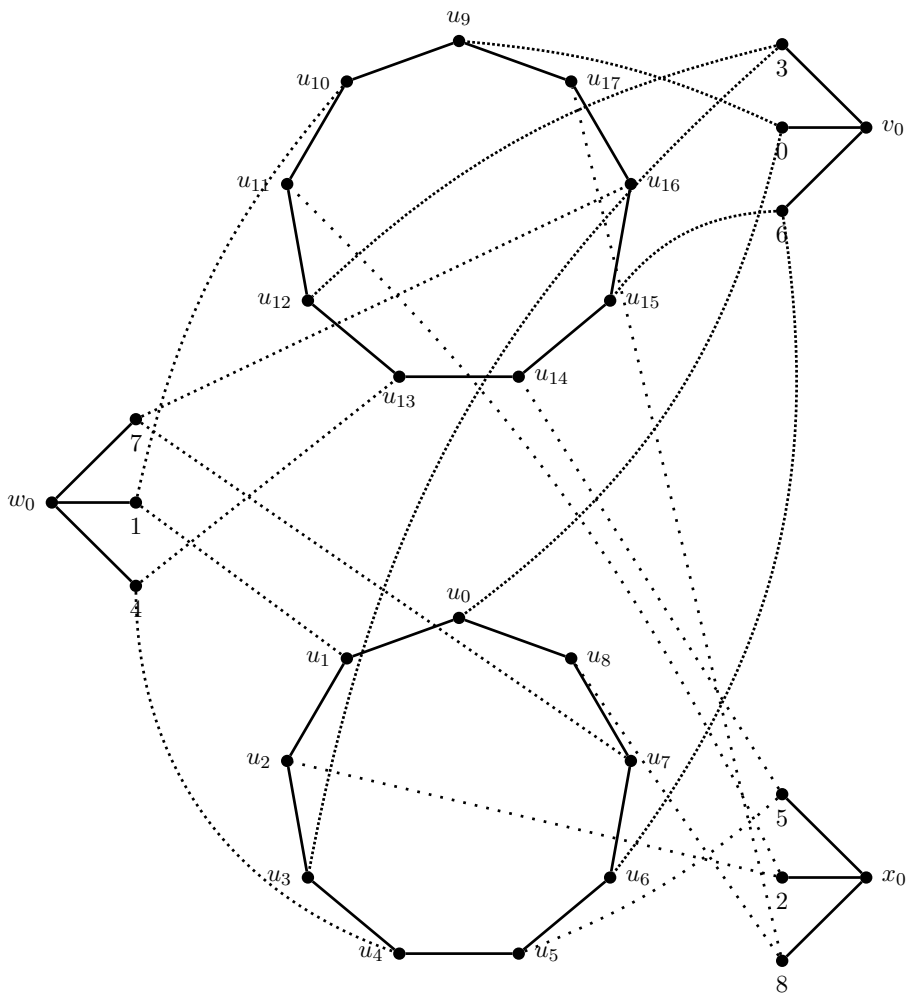

\textit{Step~2: $n$ even implies the graphs are not TF-cousins.}
When $n$ is even, $3n$ is even, so $u_m$ and $u_{m+3n}$ have
indices of the same parity. Leaf $\ell_{i,j}$ is adjacent to $u_{i+jn}$ and
$u_{i+jn+3n}$, and since $jn$ and $3n$ are both even these indices have the
parity of $i$. Colour the circuit vertices by the parity of their index, give
each leaf $\ell_{i,j}$ the colour opposite to the parity of $i$, and each
centre $c_i$ the colour of the parity of $i$. This is a proper $2$-colouring
of $\CG(n)$.
The same colouring works for $\CG'(n)$ (each copy of $C_{3n}$ is an
even cycle, since $3n$ is even, so $3n-1$ is odd and the closing edge still
joins indices of opposite parity). Thus both graphs are bipartite.

For any bipartite graph $G$, the CDC satisfies $\CDC(G)\cong G\cup G$
(the two sheets of the cover are isomorphic copies of $G$). Hence
$\CDC(\CG(n))\cong\CG(n)\cup\CG(n)$ and
$\CDC(\CG'(n))\cong\CG'(n)\cup\CG'(n)$. Both graphs are connected, so these are isomorphic if and only
if $\CG(n)\cong\CG'(n)$, which Step~1
shows is false. By
Theorem~\ref{thm:cdc}, $\CG(n)$ and $\CG'(n)$ are not TF-isomorphic.

\textit{Step~3: $n$ odd implies the graphs are TF-cousins.}
Let $n$ be odd, so that $3n$ is odd. Then $\CG'(n)$ contains the odd circuit
$C_{3n}$ and is not bipartite. Neither is $\CG(n)$: the circuit path from
$u_m$ to $u_{m+3n}$ has odd length $3n$, and closing it through the leaf
attached to that antipodal pair gives a cycle of length $3n+2$.

Let $p(m)$ denote the residue of $m$ modulo $3n$, so that $u_m$ and
$u_{m+3n}$ are the two vertices of the antipodal pair indexed by $p(m)$.
Define $(\alpha,\beta)$ on $V(\CG'(n))$ by
\[
  \alpha(u_m)=
    \begin{cases} u_m, & p(m)\ \text{even},\\ u_{m+3n}, & p(m)\ \text{odd},\end{cases}
  \qquad
  \beta(u_m)=
    \begin{cases} u_{m+3n}, & p(m)\ \text{even},\\ u_m, & p(m)\ \text{odd},\end{cases}
\]
with the subscript $m+3n$ read modulo $6n$, and $\alpha(v)=\beta(v)=v$ for every claw vertex $v$.
Both maps send each antipodal pair to itself, so both are bijections
$V(\CG'(n))\to V(\CG(n))$; and $\beta=\sigma\alpha$, where $\sigma$ is
translation by $3n$ on circuit vertices and the identity on claw vertices.

We verify the condition of Section~1: $(u,v)$ is an arc of $\CG'(n)$ if and
only if $(\alpha(u),\beta(v))$ is an arc of $\CG(n)$. Each edge contributes
two arcs, and we take them in turn.
\begin{itemize}
\item \emph{Claw arcs.} These are the arcs between a claw centre $c_i$ and
  one of its leaves $\ell_{i,j}$. Both maps fix every claw vertex, and the
  claws are the same in $\CG'(n)$ and $\CG(n)$, so each such arc is its own
  image and is an arc of both graphs. Since $c_i$ is adjacent only to
  $\ell_{i,0},\ell_{i,1},\ell_{i,2}$, every arc incident with a claw centre
  is of this kind.
\item \emph{Leaf--circuit arcs.} In either graph $\ell_{i,j}$ is joined to
  both vertices of the antipodal pair $\{u_{i+jn},u_{i+jn+3n}\}$ and to no
  other circuit vertex. Both maps fix $\ell_{i,j}$, and both send every
  antipodal pair to itself, so $\alpha(u_t)$ and $\beta(u_t)$ lie in the same
  pair as $u_t$ for every $t$. The arc $(\ell_{i,j},u_t)$ has image
  $(\ell_{i,j},\beta(u_t))$ and the arc $(u_t,\ell_{i,j})$ has image
  $(\alpha(u_t),\ell_{i,j})$; in either case the image is an arc of $\CG(n)$
  exactly when the arc it came from is.
\item \emph{Circuit arcs.} By definition $\alpha$ shifts $u_m$ by $3n$
  exactly when $p(m)$ is odd, and $\beta$ does so exactly when $p(m)$ is
  even. If
  $u_a$ and $u_b$ are consecutive in one copy of $C_{3n}$ then $p(a)$ and
  $p(b)$ have opposite parity. In the arc $(u_a,u_b)$ the map $\alpha$ is
  applied to $u_a$ and $\beta$ to $u_b$, so opposite parity means that
  $\alpha$ and $\beta$ either both shift or both do not, and the image is
  $(u_{a+3n},u_{b+3n})$ or $(u_a,u_b)$: an arc of $C_{6n}$ either way. The
  reverse arc $(u_b,u_a)$ is the same case with $a$ and $b$ interchanged. The
  arcs closing the copies have endpoints whose indices have the \emph{same}
  parity, because $3n-1$ is even, so for these exactly one of $\alpha$ and
  $\beta$ shifts. Their four images are $(u_{3n-1},u_{3n})$,
  $(u_{3n},u_{3n-1})$, $(u_0,u_{6n-1})$ and $(u_{6n-1},u_0)$, the arcs of
  $C_{6n}$ that join the copies.
\end{itemize}
Every arc of $\CG'(n)$ is thus carried to an arc of $\CG(n)$. Since $\alpha$
and $\beta$ are bijections, $(u,v)\mapsto(\alpha(u),\beta(v))$ is injective on
ordered pairs; and both graphs are cubic on $10n$ vertices, so each has
$30n$ arcs. The map is therefore a bijection from the arcs of $\CG'(n)$ onto
those of $\CG(n)$, which gives the converse implication.

The pair $(\alpha,\beta)$ is therefore a TF-isomorphism from $\CG'(n)$ to
$\CG(n)$. Together with the non-isomorphism proved in Step~1, $\CG(n)$ and
$\CG'(n)$ are TF-cousins.

\end{proof}

\begin{Rem}\label{rem:generalk}
The same construction and proof work for any integer $r\geq 2$, replacing
$3n$ by $rn$ throughout. The seed pair $C_{rn}\cup C_{rn}$ and $C_{2rn}$
satisfies $\CDC(C_{rn}\cup C_{rn})\cong\CDC(C_{2rn})$ if and only if $rn$
is odd, since the CDC of an odd circuit $C_m$ is the single circuit
$C_{2m}$ while that of an even circuit is two copies of $C_m$. When $rn$ is
even, which is forced whenever $r$ or $n$ is even, no TF-cousin
pair arises; when $rn$ is odd, the antipodal TF-isomorphism extends to the
full claw graph as before. For
$r=3$ the graphs are cubic; for $r\neq3$ the claw centres have
degree~$r$ while all other vertices have degree~$3$, so the graphs are no
longer regular. An isomorphism then carries centres to
centres since it preserves degrees, and the $6$-circuit characterisation of
Step~1 is needed only when $r=3$. The non-isomorphism argument via induced subgraph
connectivity is unchanged.
\end{Rem}

\begin{Rem}\label{rem:claw_props}
The graphs $\CG(3)$ and $\CG'(3)$ share
several invariants: both are cubic, have 30~vertices and
45~edges, are triangle-free with girth~6, have diameter~5, are bridgeless,
and contain exactly~9 $6$-circuits. Each has automorphism group of
order~$36$ with three vertex orbits: the three claw centres, the nine
leaves, and the eighteen circuit vertices. That these cannot merge is
immediate: the centres form an orbit, so their
neighbourhood, the set of leaves, is preserved, and with it the set of
circuit vertices. In particular, neither graph is vertex-transitive. The
adjacency spectra of the two graphs differ, however, confirming that they are
non-isomorphic despite sharing all the above
invariants. The spectral radius of both is~3, as expected for any cubic
graph, so it is the rest of the spectrum that separates them.

By Proposition~\ref{prop:clawgeneralisation} and
Theorem~\ref{thm:cdc}, the canonical double covers of the two graphs are
isomorphic.
Neither graph is isomorphic to any
generalised Petersen graph $\mathrm{GP}(15,j)$ (none of which has girth~$6$).
Since circulant graphs are vertex-transitive, neither is a circulant. To the best of our knowledge, neither graph appears
in any published census of named cubic graphs.
\end{Rem}

\section{Concluding Remarks}
\label{sec:conclusion}

We have introduced a unified framework of lifting and guided folding
for two problems that have typically been treated independently: constructing
unstable graphs, and constructing TF-cousins. The central observation,
formalised in Theorem~\ref{thm:stabilitytfmain} and sharpened by
Lemmas~\ref{prop:involution01}--\ref{prop:involution03} and Theorem~\ref{prop:involution04}, is that both
problems are governed by the same algebraic datum: the conjugacy classes of
strongly switching involutions in the automorphism group of the common CDC.

The language of lifting and folding, adapted from voltage graph
theory~\cite{grosstucker01}, makes this relationship transparent.
Each guide $\phi$ determines a specific fold, and distinct conjugacy classes
of guides yield distinct TF-cousins sharing the same CDC
(Lemma~\ref{prop:involution03} and
Theorem~\ref{prop:involution04}).
The constructions of Section~\ref{sec:constructions} show concretely that
both problems have solutions at every order from~$6$ on: the seed pair
$(C_k\cup C_k,\, C_{2k})$, together with the operations of adding
split-image edges and pins, yields TF-isomorphic pairs and
unstable graphs of every such order.
A further construction (Proposition~\ref{prop:substitution}) produces new
TF-isomorphic pairs by replacing entangled vertices with odd circuits,
giving a second infinite family of examples. The two members may
be isomorphic (Remark~\ref{rem:substitution-noniso}), though under the
hypotheses of Corollary~\ref{cor:substitution-family} they are not. The claw-graph
family of Section~\ref{sec:claw} demonstrates that prominent graphs,
including the Petersen graph, arise naturally within this framework.

Several questions remain open. One concerns the role of the
seed pattern: it is present in every example we have examined, and we
conjecture that this is always so.

\medskip
\noindent\textbf{Conjecture.} \emph{Let $G$ and $H$ be non-isomorphic
graphs with $\CDC(G)\cong\CDC(H)$. Then there is an odd $k$ such that one
of $G$, $H$ contains two vertex-disjoint copies of $C_k$ and the other
contains $C_{2k}$. Every unstable asymmetric graph likewise contains both
$C_k$ and $C_{2k}$ for some odd $k$.}
\medskip

\noindent The two copies of $C_k$ and the circuit $C_{2k}$ must be allowed
to lie in different members. Requiring a single graph to contain both $C_k$
and $C_{2k}$ fails already for the Petersen graph, whose cycle lengths are
$5$, $6$, $8$ and $9$: it has neither a triangle nor a $C_{10}$, so no odd
$k$ serves. The hypothesis that the unstable
graph is asymmetric is equally necessary. Take $C_6$ with a pin
attached to an antipodal pair, the upper graph of
Figure~\ref{fig:sevenvertex}(a). It is connected, non-bipartite and
vertex-determining, and is unstable with index of instability $2$, yet its
only cycle lengths are $5$ and $6$, so again no odd $k$ serves. Its
automorphism group has order $4$, so it is not asymmetric and the second
statement does not apply to it.

Srivastava~\cite{srivastava} gives a second such graph, offered
there as a counterexample to a conjecture in an earlier version of this
paper. It is connected, non-bipartite and vertex-determining, on ten
vertices, with the eleven edges
\[
  \{1,2\},\ \{1,10\},\ \{2,4\},\ \{3,4\},\ \{3,6\},\ \{4,5\},\ \{5,7\},\
  \{5,10\},\ \{6,10\},\ \{7,9\},\ \{8,10\};
\]
it is unstable, the pair $\alpha=(1\ 6)$, $\beta=(2\ 3)$ being a non-trivial
TF-automorphism, and its simple circuits have lengths $5$, $5$ and $6$, so no
odd $k$ gives it both $C_k$ and $C_{2k}$. Its automorphism group has order
$2$, generated by $\alpha\beta$; so, like the seven-vertex graph above, it is
not asymmetric, and the hypothesis stated here is what excludes both. We are
grateful to Srivastava for the example.

The circuits cannot instead be
located in the common CDC. A bipartite graph contains no odd circuit, so the
cover cannot carry the copies of $C_k$. As for $C_{2k}$, its presence in the
cover follows from non-bipartiteness of the base graph, which our
standing hypotheses already impose.

We have verified the conjecture for all connected graphs on at most $9$
vertices using McKay's complete enumeration~\cite{mckaydata,mckay}, checking
CDC isomorphism directly in each case. The counts of connected
non-isomorphic TF-cousin pairs are: no pairs on $6$ vertices, $3$ on $7$,
$24$ on $8$, and $255$ on $9$. In every case the
conjecture holds with $k=3$.

The conjecture requires only some odd $k$, and $k=3$ is not always
available. Attaching a pin to an entangled pair of the seed pair
$(C_5\cup C_5,\,C_{10})$, as in Construction~\ref{con:pins}, gives a
TF-cousin pair on eleven vertices in which neither graph contains a
triangle. Here the conjecture is satisfied by the two $5$-circuits of one
member and the $10$-circuit of the other.

The statement about unstable asymmetric graphs
is vacuous for graphs on at most nine vertices, since the smallest known
unstable asymmetric graph has twelve vertices
(Figure~\ref{fig:gammacockroach}). For that graph the conjecture
holds with $k=3$ and again with $k=5$. J.~Goodchild has confirmed the conjecture
independently, by an implementation of his own, for all connected graphs of
maximum degree at most $3$ on at most $13$ vertices, and for a further $420$
TF-cousin pairs on $14$ to $26$ vertices; no order in that second range was
searched exhaustively.

Collins and Sciriha~\cite{collins2020} ask whether graphs with
isomorphic CDCs necessarily share the same main eigenvalues, reporting that
no counterexample exists on at most eight vertices. The question remains
open.

The appendix of~\cite{collins2020} lists $32$ pairs. These are the
pairs on exactly eight vertices, obtained by filtering McKay's $12\,346$
graphs of that order, with no restriction on connectedness, bipartiteness or
vertex-determining properties; a pair on fewer than eight vertices appears
there with isolated vertices added to both members, so the $32$ account for
all pairs on at most eight vertices. It is explicitly noted
in~\cite{collins2020} that $\mathrm{CDC}(C_6)\cong\mathrm{CDC}(K_3\cup K_3)$,
an instance of a pair with a disconnected member. Our enumeration uses
McKay's connected graph files~\cite{mckaydata,mckay} and therefore covers
connected graphs only. Our $24$ connected pairs on eight vertices form a
subset of the $32$; the remaining $8$ have a disconnected member, three of
them being the seven-vertex pairs above with an isolated vertex added.

The number of TF-cousins of a given graph is determined in principle by the
conjugacy classes of strongly switching involutions of its CDC
(Theorem~\ref{prop:involution04}), and by Bychawski's refinement
(Remark~\ref{rem:bychawski}) the cardinality of each class gives the index
of instability of the corresponding base graph.
Proposition~\ref{prop:clawgeneralisation} establishes that $\CG'(n)$ is a TF-cousin of $\CG(n)$ for each odd $n$.
For $n\geq 3$ it is not the only one. Reversing the circuit of
$\CG(n)$ about an even position, and carrying the leaves and centres with
it, gives an automorphism $\alpha$ under which no vertex is adjacent to its
image; the map $(x,\varepsilon)\mapsto(\alpha(x),1-\varepsilon)$ is
therefore a guide for $\CDC(\CG(n))$: condition (G1) of
Definition~\ref{def:guide} holds because $\alpha$ is an automorphism, and
condition (G2) because no vertex is adjacent to its image. Folding along this guide produces a
further base graph, again cubic and connected, but containing triangles and
hence isomorphic to neither $\CG(n)$ nor $\CG'(n)$, so that $\CG(n)$ has at
least two TF-cousins. This was communicated to the
author by V.~Rozho\v{n} and R.~\v{S}\'{a}mal, having been obtained with the
Bolzano system~\cite{bolzano}. For $n=1$ this fold returns $\CG'(1)$
itself, and the Petersen graph has exactly one TF-cousin.
Determining the number of TF-cousins of $\CG(n)$ for general odd $n$ remains
open, and would be resolved by a complete description of the strongly
switching involutions of $\CDC(\CG(n))$.

The framework developed here and the group-theoretic
constructions of Bychawski~\cite{bychawski} approach the same objects from
opposite directions. His families are obtained by existence arguments,
realising prescribed automorphism groups and prescribed numbers of
companions; the graphs are described through those groups rather than
exhibited individually. Here, by contrast, each TF-cousin is produced
explicitly, by folding along a guide. Two questions about the interaction of
the two approaches seem worth pursuing. The first concerns the guides for
his families. The bijection of his Theorem~8.6 guarantees that they exist;
what is not clear is how to describe them, since his construction reaches
the graphs through their automorphism groups rather than by folding. The second
concerns the conjecture above:
his construction produces TF-cousin pairs directly, at orders where checking
every graph is no longer feasible, and such pairs would make natural test cases.

Progress on any of these questions would show more clearly how the
algebraic structure of $\Aut(\CDC(G))$ governs the combinatorial
properties of the base graphs.

\section*{Acknowledgements}
The author thanks Josef Lauri and Raffaele Scapellato for their encouragement
and guidance during the early stages of this work. The term \emph{TF-cousin}
was used informally in our earlier collaboration; it is adopted here in the
hope that they find it a fitting name. The author is grateful to
Jacob Goodchild, who verified the computational claims of this paper
independently, extended the verification of the conjecture of
Section~\ref{sec:conclusion} well beyond the published range, and drew
attention to the counterexamples that led to its present formulation. He
worked with an implementation of his own, developed with AI assistance under
his direction. The author also thanks the two anonymous referees,
whose careful reading and detailed reports led to a number of corrections
and improvements throughout the paper.

\section*{Declaration on the use of AI}
In preparing this work the author used a large language model in two ways: to polish the wording of the text, and to help write
and debug the Python code by which the computational claims made here were
obtained and checked, using \texttt{networkx} together with McKay's published
catalogues~\cite{mckaydata,mckay}. The mathematics, the statements
and the proofs are the author's own. The author reviewed and edited all such
output and takes full responsibility for the content of this paper.

\bibliographystyle{plain}
\bibliography{reference}

\end{document}